\documentclass[11pt,reqno]{amsart}
\usepackage{pgfplots}
\usepgfplotslibrary{patchplots}
\usetikzlibrary{patterns, positioning, arrows}
\usepackage{mathrsfs,amsmath,graphicx,color,bm}
\usepackage{amsmath,mathrsfs,amsfonts,amssymb}
\usepackage{cases}
\usepackage{mathtools}
\usepackage{stmaryrd}
\usepackage{color}
\usepackage{hyperref}
\numberwithin{equation}{section}

\definecolor{cmu}{RGB}{128,18,18}
\newtheorem{theorem}{{\bf Theorem}}[section]
\newtheorem{definition}[theorem]{Definition}
\newtheorem{corollary}[theorem]{{\bf Corollary}}
\newtheorem{lemma}[theorem]{{\bf Lemma}}
\newtheorem{prop}[theorem]{{\bf Proposition}}

\newcommand{\e }{\varepsilon}
\newcommand{\p}{\partial}
\renewcommand{\O}{\Omega}
\renewcommand{\div}{\operatorname{div}}

\newcommand{\dist}{\mathrm{dist}}
\newcommand{\tr}{\operatorname{tr}}
\renewcommand{\leq}{\leqslant}
\renewcommand{\geq}{\geqslant}

 \newcommand{\bxi}{\boldsymbol{\xi}}
 \newcommand{\bnu}{\boldsymbol{\nu}}
  \newcommand{\bphi}{\boldsymbol{\varphi}}
  \newcommand{\bgamma}{\boldsymbol{\gamma }}
\newcommand{\HH}{\mathbf{H}}

\newcommand{\1}{\mathbf{1}}
\renewcommand{\parallel}{\mathbin{\!/\mkern-5mu/\!}}

\newcommand{\R}{\mathbb{R}}
\newcommand{\nn}{\mathbf{n}}
\newcommand{\mm}{\mathfrak{m}}
\newcommand{\uu}{\mathbf{u}}
\newcommand{\cc}{\mathcal{C}}
\newcommand{\dd}{\mathrm{d}}
\newcommand{\vv}{\mathbf{v}}
\newcommand{\pp}{\mathbf{p}}

\renewcommand{\tilde}{\widetilde}

\newcommand{\mres}{\mathbin{\vrule height 1.6ex depth 0pt width
0.13ex\vrule height 0.13ex depth 0pt width 1.3ex}}
\def\({\left(}
\def\){\right)}

\newcommand{\picdis}[1]{#1}
 \def\[{\begin{equation}}
 \def\]{\end{equation}}
\begin{document}

\title{Phase Transition  of Parabolic Ginzburg--Landau Equation   with  Potentials of High-Dimensional Wells}

 \author{Yuning Liu }
\address{NYU Shanghai, 567 Yangsi W road, Pudong, Shanghai 200126, China, and NYU-ECNU Institute of
Mathematical Sciences at NYU Shanghai, 3663 Zhongshan Road North, Shanghai, 200062, China}
\email{yl67@nyu.edu}

\begin{abstract} 
In this work, we study the co-dimensional one  interface limit  and geometric motions of  parabolic   Ginzburg--Landau  systems       with  potentials of high-dimensional wells.
The main result generalizes the one by  Lin et al. (Comm. Pure
Appl. Math., 65(6):833-888, 2012) to a dynamical case.
In particular combining      modulated energy methods and weak convergence methods,  we derive the    limiting   harmonic heat flows   in the inner and outer bulk regions segregated by the sharp interface, and a non-standard boundary condition for them. These results are valid   provided that the initial datum of the system is  well-prepared  under   natural energy assumptions.  

\end{abstract}

 \maketitle

%\date{\today}

%\tableofcontents

%%%%%%%%%%%%%%%%%%%%%%%%%%%%%%%%%%%%%%%%%%%%%%%%%%%%%%%%%%%%%%
%%                                                            %%
%% I N T R O D U C T I O N                     %%
%%                                                            %%
%%%%%%%%%%%%%%%%%%%%%%%%%%%%%%%%%%%%%%%%%%%%%%%%%%%%%%%%%%%%%%

\section{Introduction}
In  the   work of Keller--Rubinstein--Sternberg \cite{MR978829,MR1025956}, they  created    a general gradient flow theory     in the descriptions of     fast reaction and  slow diffusion,  and established  its relations to    the mean curvature flow (MCF) and the harmonic heat flows into manifolds. These works involve  some formal statements associated with the multiple components phase transitions with higher dimensional wells. Such statements are also refereed to as the {\it Keller--Rubinstein--Sternberg problem}. More precisely they investigated  the vectorial Allen--Cahn equation (also called Ginzburg--Landau equation)
\[\p_t \uu_\e=\Delta \uu_\e-\e^{-2} \p F(\uu_\e),\label{GL introduction}\]
where $\uu_\e(x,t):\O\times (0,T) \mapsto \R^n$ is a mapping depending on a small parameter $\e>0$ and $\O\subset \R^d$ is a bounded domain with $C^1$ boundary. Here    $F(\uu)$ is a       double  equal-well   potential with ground state being the disjoint union of two  smooth closed submanifolds $\mm_\pm \subset \R^n$, and $\p F(\uu)$ is the differential of $F$ at $\uu$. The {\it Keller--Rubinstein--Sternberg problem} is  concerned with the limiting behavior of $\uu_\e$ as $\e$ tends to zero.
In the work of Lin--Pan--Wang \cite{Lin2012a}, they  set up an analytic  program to  
rigorously justify the formally asymptotic  analysis given in the aforementioned works of Keller--Rubinstein--Sternberg. To be more precise 
they considered  the  minimizers   of  the Ginzburg--Landau functional
\[A_\e(\uu_\e):=\int_\O\( \frac{\e }2 |\nabla \uu_\e |^2+\frac 1\e  F(\uu_\e )\)\, dx\label{GL energy}\] 
 that satisfy  well-prepared  boundary conditions on $\p\O$. They  established  the co-dimensional one   interface limit  of \eqref{GL energy}, which essentially generalizes   the $\Gamma$-convergence of  Modica--Mortola \cite{MR0445362}
 to    vectorial cases (though they did not state their main theorems in such an abstract manner). More importantly, they showed   that the limits  of $\uu_\e$ in the bulk regions  correspond  to  minimizing harmonic maps into $\mm_\pm$, and  they   derived  a non-standard boundary condition (also called the minimal pair condition, cf. \eqref{thm minimal pair} below) which services as  a constraint  on   the limiting harmonic maps when restricted on the interface. Such a boundary condition is a  new feature that arises  due to minimization of   surface tensions  in vectorial cases. Note that such a condition   holds trivially in the   case of  scalar Allen-Cahn equation.
 % because in this case  the solution mappings will get stuck at the ground states, namely two distinct points in $\R^1$, in the limit. 
%Note that this later work gives the first proof  of phase transition in the sense of   $\Gamma$-convergence. See a introductory notes by Alberti \cite{MR1757697}.
 
In this work we shall try to  generalize Lin--Pan--Wang \cite{Lin2012a} to the parabolic system \eqref{GL introduction}  by  proving  the following statements: Firstly, for well-prepared initial datum, as $\e$ tends to $0$, the solution gradients of  \eqref{GL introduction} will undergo phase transitions 
 across  a moving interface $\Sigma_t$ that propagates  according to  (two-phase) MCF. Secondly, in the two bulk regions $\O_t^\pm$ segregated by the interface $\Sigma_t$,  the solutions will converge to  harmonic heat flows mapping into $\mm_\pm$ respectively.  Finally,  the one-sided traces of  the limiting harmonic heat flows on $\Sigma_t$ must satisfy the minimal pair  condition   firstly derived in \cite{Lin2012a}.

  Our first result is a vectorial analogy of   the co-dimensional one  scaling limit of   scalar parabolic Allen--Cahn  equation to the  MCF, i.e. the special case of  \eqref{GL introduction} when   $\mm_\pm$ are  two distinct points $a^\pm\in  \R^1$. There have been major progresses  in the  scalar  case over  the last thirty  years, made under   different frameworks. Here we   mention two classes of results  and leave the  discussions of  some  others in the sequel.  One is  the convergence to a Brakke's flow by Ilmanen \cite{MR1237490} using a version of Huisken's monotonicity formula together with tools from geometric measure theory. See also \cite{MR1425577,MR1803974,MR2040901,MR2253464,MR3495430,MR2440879} and the references therein for further renovations. Despite of its energetic nature,  a major difficulty of such an approach is   the control of the so called {\it  discrepancy measure}, and in every existing literature in this direction  the method     relies   crucially  on a version of   Modica's maximum principle \cite{MR803255}. 
  There have been attempts to generalize such a method to   vectorial cases. However, it is not clear  whether  Modica's maximum principle holds for  elliptic system. 
  Another approach, which relies more  on the parabolic comparison principle,  is  the global in time  convergences  to the  viscosity solution of  MCF. These are weak solutions to the MCF built independently by    Chen--Giga-Goto \cite{MR1100211} and Evans--Spruck   \cite{MR1100206}. Concerning the  convergence of scalar Allen--Cahn equation to such solutions, we refer the readers to the work of Evans--Soner--Souganidis \cite{MR1177477}, the work  of Soner \cite{MR1674799} and   the references therein.  These two approaches both give global in time (weak) convergences to   weakly defined  solutions of   MCFs  up to their life spans. However, as  their technics  involve  parabolic maximum principle and   comparison principle in one way or another,  it is not clear how to use them to attack   vectorial cases  in general. It is worth mentioning that for radially symmetric initial datum and when $\mm_\pm$ are two concentric circles, Bronsard--Stoth \cite{MR1443865} obtain  global in time convergences to MCF of planar circles.

 To the best of our  knowledge, there are mainly two approaches   to rigorously justify   the convergences of  the vectorial Allen--Cahn equations, both  assuming that   the limiting interface propagation  problem  has a (local in time)  classical solution. Compared with the aforementioned methods which lead to global in time (weak) convergences, they have quite different natures.
 One of these methods is the asymptotical expansion technics  developed by De Mottoni--Schatzman \cite{MR1672406} and  later by Alikakos--Bates--Chen \cite{MR1308851}, which has been  used  recently  in  \cite{Fei2023aa,MR4059996} for matrix-valued cases of \eqref{GL introduction}. 
  In particular, Fei--Lin--Wang--Zhang \cite{Fei2023aa} studied  the case when $\mm_\pm=O^\pm(n)$, the $n$-dimensional orthogonal group. By inner-outer expansions together with a gluing procedure,  such an approach reduces the convergence problem to a linear stability problem given that the limiting system (not merely  the limiting interface motion) is strongly well-posed. 
  The major challenge of this approach is  the analysis of the spectrum of the linearized operator at the  {\it minimal orbits} (also called optimal profile) or their variants. 
  Indeed, one of the novelties  of   \cite{Fei2023aa} is to devise the so called {\it quasi-minimal orbits}, overcoming the lack of {\it minimal orbits} in the bulk regions, and to derive the spectrum stability of the linearized operator at such orbits making use of the   minimal pair  condition.
  Finally we refer   Lin--Wang \cite{MR4002307} for a general theory for the strong well-posedness of the   limiting system.
  
Another approach, which also assumes a regular solution of the limiting interface motion but not the limiting  harmonic heat flows,   is the relative entropy  method developed  by Fischer--Laux--Simon  \cite{fischer2020convergence}, motivated by Jerrard--Smets \cite{MR3353807} and Fischer--Hensel \cite{MR4072686}. A generalization to matrix--valued  case has been done by Laux-Liu  \cite{MR4284534} to study the isotropic--nematic transition in Landau--De Gennes model of liquid crystals. More recently, in \cite{liu2021sharp} the author used these methods, together with those developed   by   Lin--Wang \cite{lin2020isotropic}, to investigate  the convergence problem of  an anisotropic 2D Ginzburg--Landau model.
% In particular,  he derived  some delicate convergence results  of the level sets of the solutions, which are crucial to obtain    anchoring boundary conditions of the limiting solutions.  

Now we  introduce a minimum amount of  terminologies   necessary for stating  the main result of this work.
 Let 
 \[ \mm_\pm\text{  be two disjoint   smooth,   closed,  connected  submanifold   in } \R^n.\label{mm assumption}\] 
For technical purposes we assume $0\in\mm_-$. 
Let   $F:\R^n\to [0,\infty)$ be a smooth   function  with $\mm_\pm$ being its double equal-wells: 
\[\operatorname{Arg~min} F= \mm:=  \mm_+\sqcup  \mm_-. \label{limit manifold}\]
We assume that $F(\uu)$ only depends on the distance from $\uu$ to $\mm$. That is,  
\[F(\uu)=  f(\dd_\mm^2(\uu)),\label{bulk potential}\] 
where $\dd_\mm(\uu)$ is the  distance (see \eqref{dN global} below for the full definition), and $ f(s)  \in C^2(\R^+, \R^+)$ satisfies  
  \begin{align} \label{bulk2} 
  \begin{cases}
  f(s) =  s& \text{ if } 0\leq s\leq \delta_0^2,\\
 f(s)= 2\delta_0^2    & \text{ if  } s\geq 2\delta_0^2.
  \end{cases} 
  \end{align} 
In \eqref{bulk2}  $\delta_0>0$ is a small number so that  the nearest-point projection $P_\mm$ from $B_{2\delta_0}(\mm)$, the $2\delta_0$-tubular neighborhood of $\mm$,  to  $\mm$ is smooth.
  
 We consider  the following  initial boundary value problems  on a bounded domain  $\O\subset \R^d$   with $C^1$ boundary:
\begin{subequations}\label{Ginzburg-Landau sys}
\begin{align}
\partial_{t} \uu_\e &= \Delta \uu_\e  -  \e ^{-2}\p   F(\uu_\e )&&~\text{in}~ \Omega\times (0,T),\label{Ginzburg-Landau}\\
\uu_\e  &=\uu_{\e}^{in} &&~\text{in}~\Omega\times \{0\},\\
\uu_\e &=\mathbf{g} &&~ \text{on}~\p\O\times (0,T).\label{bc of omega}
\end{align}
 \end{subequations}
Here   $\p   F(\uu)$ is the gradient  of $F(\uu)$, and  $\mathbf{g}:\overline{\O}\mapsto \mm_-$ is a given  smooth  mapping.
 Our main result is concerned with the asymptotical  behaviors  of solutions to \eqref{Ginzburg-Landau sys} for well-prepared initial datum. To give an analytic characterization  of such initial datum,  we need to set up the geometry of the interface motion. To this end, we  assume  that 
\begin{equation}\label{interface}
 \Sigma=\bigcup_{t\in [0,T]}\Sigma_t \times \{t\}~\text{is a smoothly evolving closed hypersurface in}~\O,
\end{equation}
 starting from a closed smooth surface $ \Sigma_0\subset  \O$. We denote by  $\O^\pm_t$   the domain segregated  by $\Sigma_t$, and   by 
 \begin{align}\label{signed dist}
 \dd_\Sigma(x,t) \text{   the signed-distance  from } x \text{ to the set } \Sigma_t  \text{ taking  positive  values in  }\O^+_t, 
 \end{align}
 and taking   negative values  in $\O^-_t=\O\backslash \overline{\O^+_t}$. In other words, 
 \begin{equation}\label{def:omegapm}
\Omega^{\pm}_t:=  \{x\in\Omega\mid \dd_\Sigma(x,t)\gtrless0\}.
\end{equation}
To avoid contact angle problems,  we assume that  $\Sigma$ stays at least $4\delta_0$ distant away  from   $\p\O$.  

 Following \cite{MR3353807,MR4072686,fischer2020convergence}, we define  the modulated energy (also called the  relative entropy energy) by 
\begin{align}
\label{entropy}
E_\e  [\uu_\e  | \Sigma](t) &:=   \int_\O \(\frac{\e}{2}\left|\nabla \uu_\e (\cdot,t)\right|^2+\frac{1}{\e} {F (\uu_\e (\cdot,t))}-  \bxi \cdot\nabla \psi_\e (\cdot,t) \)\, dx.
\end{align}
Here  $\bxi$ is an appropriate extension of the unit normal vector field  of $\Sigma$ (see \eqref{def:xi} below), and $\psi_\e$ is the 
scalar function
\begin{align}
\psi_\e (x,t):=  \dd_F \circ \uu_\e (x,t) \label{psi}
\end{align}
with   $\dd_F$   defined  by \eqref{quasidistance} below. As we shall see later on, the integrand of \eqref{entropy} is non-negative, and enjoys several   coercivity estimates including  controls  of       {\it discrepancy}   and    calibration  of the Ginzburg--Landau energy \eqref{GL energy}.
We also need the surface tension coefficient
\[c_F  :=2 \int_{0}^{\frac {\dist_{\mm}}2} \sqrt{2f(\lambda^2)} d \lambda,\label{linpanwang cf equ}\]
where  
$\dist_{\mm}$
is  the Euclidean distance between $\mm_+$ and $\mm_-$, 
and another modulated energy controlling the bulk errors:
\[ B[\uu_\e  | \Sigma](t):= \int_\O   \Big(c_F\chi-c_F+ 2(\psi_\e-c_F)^- \Big)\eta\circ \dd_\Sigma  \, dx+\int_\O  \( \psi_\e-c_F\)^+|\eta\circ\dd_\Sigma| \, dx.\label{gronwall2new}\]
In \eqref{gronwall2new}  $\chi(\cdot,t)=\1_{\O_t^+}-\1_{\O_t^-}$ and $g^\pm$ denotes the positive/negative parts of a function $g$ respectively, and $\eta$ is an appropriate  truncation of the identity function (cf. \eqref{truncation eta}). In particular,   $( \eta\circ \dd_\Sigma)\chi\geq 0$ holds in $\O$ due to our convention on  the signed-distance function, and thus  the two  integrands in \eqref{gronwall2new} are both  non-negative. We refer the readers to     the proof of Theorem \ref{thm volume convergence} below for more details on the positivity of \eqref{gronwall2new}.

The main result of this work is the following: 
\begin{theorem}\label{main thm}
Assume that the family of hypersurfaces  $\Sigma$   \eqref{interface} evolves by    mean curvature flow during $[0,T]$. If  the initial datum of  \eqref{Ginzburg-Landau sys} is well-prepared in the sense that  
\begin{equation}\label{initial}
\e\|\uu_\e(\cdot,0)\|_{L^\infty}+B[\uu_\e  | \Sigma](0)+E_\e  [\uu_\e  | \Sigma](0)\leq C_1\e
\end{equation}
for some constant $C_1$ that is independent of $\e $,  then  there exists $C_2$ independent of $\e$ so that 
\begin{subequations}
\begin{align}
 &\sup_{t\in [0,T]} E_\e  [\uu_\e  | \Sigma](t)\leq C_2\e,\label{intro cali}\\
 &\sup_{t\in [0,T]}B[\uu_\e  | \Sigma](t)\leq C_2\e,\label{intro modulodist}\\
&\sup_{t\in [0,T]}\int_\O| \psi_\e-c_F \1_{\O_t^+} |     \, dx\leq C_2\e^{1/2}.\label{volume convergencethm}
%&\sup_{ t\in [0,T]}\left|\int_\O\left(\frac {\e}2\left|\nabla \uu_\e  \right|^2+ \frac 1{\e} F\left(\uu_\e  \right) \right) \,  d x-c_F\mathcal{H}^{d-1}(\Sigma_t)\right|\leq C_2\e^{1/2}.\label{intro energy conv}
\end{align}
\end{subequations}
Moreover,   for some subsequence  $\e _k\downarrow 0$ there holds 
  \begin{equation}\label{strong global of Q}
  \uu_{\e _k}\xrightarrow{k\to\infty }   \uu^\pm ~\text{weakly  in}~  L^2(0,T;H^1_{loc}(\O^\pm_t)),
\end{equation}
where   $\uu^\pm$ are   weak solutions to the  harmonic   heat flows into $\mm_\pm $ respectively and 
\begin{equation}\label{reg limit}
\uu^\pm \in L^\infty\(0,T;H^1( \Omega^\pm_t;\mm_\pm)\),\quad \p_t \uu^\pm \in L^2\(0,T; L^2_{loc}(\O^\pm_t)\).
\end{equation} 
Furthermore,  for   every $t\in (0,T)$,  
\[|\uu^+-\uu^-|_{\R^n}(x,t)=\dist_\mm \quad  \text{ for }\mathcal{H}^{d-1}\text{- a.e.  }  x \in \Sigma_t.\label{thm minimal pair}\]
\end{theorem}  
 
A few comments are in order.  Firstly, in   \eqref{initial} the $L^\infty$ bound of the initial datum   is used (together with \eqref{bulk2})  to obtain  an uniform in space-time  $L^\infty$-bound of $\uu_\e $, i.e.
\[\|\uu_\e \|_{L^\infty(\Omega\times(0,T))}\leq c_0\label{L infinity bound1}\]
for some $\e$-independent  constant $c_0$. Such an estimate, derived by applying the  maximum principle to \eqref{Ginzburg-Landau},   enables   us to avoid  several technical complications in the passage of the limit $\e\downarrow 0$. Indeed, even in the case when $d=2$, severe   difficulties   arise in the  anisotropic model  considered in  \cite{liu2021sharp} where an estimate like \eqref{L infinity bound1} is not available.
Secondly, if  we denote the second fundamental forms of $\mm_\pm$ at points $\pp^\pm$ by $A^\pm(\pp^\pm)(\cdot,\cdot)$, respectively,
then the  theorem above claims that  the pair of mappings
\[\uu^\pm(\cdot, t): \Omega^\pm_t \mapsto \mm_\pm\subset \R^n\label{upm mapping}\]
  satisfy  the following system in the weak sense:

\[\label{twophaselimit}
\left\{
\begin{split}
\partial_{t} \uu^\pm -\Delta \uu^\pm&=   A^\pm(\uu^\pm) (\nabla \uu^\pm,\nabla\uu^\pm)  &&~\text{in}~ \bigcup_{t\in [0,T]}\Omega^\pm_t\times \{t\},\\
|\uu^+-\uu^-|_{\R^n} (x,t)  &= \dist_\mm &&\text{ for }~\mathcal{H}^{d-1}\text{- a.e  }x\in \Sigma_t,\\
\uu^- &=\mathbf{g}, &&~ \text{on}~\p\O.  
\end{split}
\right.
\]
The first equation in \eqref{twophaselimit} says that  $\uu^\pm$ are (weak) harmonic map heat flows   from the moving domains $\Omega^\pm_t$ to  the target manifolds  $\mm_\pm$ respectively. The second equation in \eqref{twophaselimit} is   referred to as the  minimal pair boundary condition (cf.    \cite{Lin2012a,MR4002307}), and  in the last equation $\mathbf{g}:\overline{\O}\mapsto \mm_-$ is a prescribed  smooth mapping. 

To make the main theorem applicable, we shall show that the class of initial datum fulfilling the condition \eqref{initial} is geometrically rich. This is stated in the following result.
\begin{theorem}\label{thm init}
For any $\delta\in (0,\delta_0)$ and
 any pair of mappings  $\uu^{in}_\pm\in H^1(\O_0^\pm;\mm_\pm)$  with 
\[|\uu^{in}_+-\uu^{in}_-|_{\R^n} (x)  = \dist_\mm, \quad\text{ for } \mathcal{H}^{d-1}\text{- a.e }~x\in \Sigma_0,\label{MC initial data}\] 
  there exist  $\uu^{in}_\e\in H^1(\O;\R^n)\cap L^\infty(\O)$ and a constant $C=C(\delta, \uu^{in}_\pm)$  so that 
\begin{subequations}
\begin{align}
\uu_\e^{in} &=\uu^{in}_{\pm}~\text{ in }~\O_0^\pm\backslash B_{2\delta}(\Sigma_0),\label{u coincide}\\
E_\e  [\uu_\e^{in}  | \Sigma_0]   & \leq C \e,\label{u cali}\\
B[\uu_\e^{in}  | \Sigma_0] &\leq C\e.\label{u bulk}
\end{align}
\end{subequations}
\end{theorem}
  
 The rest of the work will be organized as follows: in Section \ref{sec pre}, we shall recall fundamental  results  that will be employed throughout  the work. These include  the compactness and closure of special function with bounded variation (cf. \cite[Chapter 4]{MR1857292}), the theory of minimal connection developed by Sternberg \cite{MR930124} and Lin--Pan--Wang \cite{Lin2012a},  the elements of differential  geometry used in the description of    interface motion, and finally  the relative entropy  method by Fischer--Laux--Simon \cite{fischer2020convergence}. In particular, 
 in Subsection \ref{sec entropy}, we shall adapt this later method to     system \eqref{Ginzburg-Landau sys}, and then  derive a differential inequality, i.e. Proposition \ref{gronwallprop}.    This proposition, when combined with  Chen--Struwe \cite{MR990191} along with   results   in Section \ref{sec level}, leads to the convergences  to  harmonic heat flows  locally away from the moving interface $\Sigma_t$. Another important consequence  of Proposition \ref{gronwallprop} is an     $L^1$-convergence rate estimate of $\psi_\e$, obtained  in Theorem \ref{thm volume convergence}. This theorem will be used 
to derive fine   estimates   of the  level sets of $\psi_\e$ in  Lemma \ref{area control}, as well as convergences of some corrections of $\uu_\e$ up to the moving interface $\Sigma_t$. All of these will be done in Section \ref{sec level}, and we shall use them to derive the minimal pair boundary condition \eqref{thm minimal pair}  in  Section \ref{sec mp}, and thus finish the proof of Theorem \ref{main thm}. Finally  we prove Theorem \ref{thm init}  in Section \ref{sec initial data}. 

We end the introductory part by introducing  some  notations and conventions that will be employed throughout this work.
Unless specified otherwise $C>0$ is  a generic constant  depending  only  on the geometry of the interface $\Sigma$ (cf. \eqref{interface}) and that of  the wells $\mm$ (cf. \eqref{mm assumption}), but not on $\e$ or $t\in [0,T]$. The value of such a constant   might change from line to line. In order to simplify the presentation, we shall sometimes  abbreviate the estimates like  $X\leq CY$ by $X\lesssim Y$ for two  non-negative quantities $X,Y$.

We provide  a list of symbols for the convenience of the readers:

  \begin{itemize}
  \item $A:B$  is the    Frobenius inner product of   two square matrices $A,B$,   defined by  $  \tr A^{\mathsf{T}} B$.
  \item $\p_i=\p_{x_i} ~(0\leq i\leq d)$ with the convention that $\p_t=\p_0$. 
  \item $\nabla f=(\p_1f, \cdots, \p_d f)$ is the (distributional)  gradient of a function $f$ with variables $x=(x_1,\cdots,x_d)$.
\item $\p W=(\p_{u_1} W,\cdots,\p_{u_n} W)$ is the   gradient of   a   function   $W=W(\uu)$.
    \item $\p \dd_F(\uu)$: the generalized gradient of   $\dd_F$ (cf. \eqref{def of linear map in chain} below).
  \item  $\p  U$: measure-theoretic boundary of a set $U$ of finite perimeter with measure-theoretic outer normal vector $\nu$.
   \item $\dist(\uu,A)$:  Euclidean distance from $\uu $ to a set $A\subset \R^n$.
   \item For two vectors $\uu,\vv\in\R^n$, $|\uu-\vv|$ is their Euclidean distance $\left|\uu-\vv\right|_{\R^n} $.
   \item $\dist_{\mm}$: the distance between $\mm_\pm$ in $\R^n$, namely  $\dist_{\mm}:=\inf_{\pp^{\pm} \in \mm_\pm}\left|\pp^{+}-\pp^{-}\right|_{\R^n} $.
  \item $\dd_\mm(\uu)$:  the distance    from $\uu\in\R^n$ to $\mm=  \mm_+\sqcup  \mm_-$ (cf.   \eqref{dN global} below).

  \item $B_\delta(U)$: the $\delta$-(tubular) neighborhood of a set $U$ in the corresponding Euclidean space. In particular, $B_\delta(x)$ is the open ball centered at $x$.
  \item 
 $\dd_\Sigma(x,t)$: signed distance from $x$ to $\Sigma_t$  (cf. \eqref{signed dist}).
     \end{itemize}

 %%%%%%%%%%%%%%%%%%%%%%%%%%%%%%%%%%%%%%%
   %%%%%%%%%%%%%%%%%%%%%%%%%%%%%%%%%%%%%%% %%%%%%%%%%%%%%%%%%%%%%%%%%%%%%%%%%%%%%% %%%%%%%%%%%%%%%%%%%%%%%%%%%%%%%%%%%%%%% %%%%%%%%%%%%%%%%%%%%%%%%%%%%%%%%%%%%%%% %%%%%%%%%%%%%%%%%%%%%%%%%%%%%%%%%%%%%%% %%%%%%%%%%%%%%%%%%%%%%%%%%%%%%%%%%%%%%%

  %%%%%%%%%%%%%%%%%%%%%%%%%%%%%%%%%%%%%%%%%%%%%%%%%%%%%%%%%%%%%%%%%%%%%%%%%%%%%%%%%%%%%%%%%%%%
\section{Preliminaries}\label{sec pre}
 
 \subsection{Special function of bounded variation}
 
\begin{definition}
We say that $\uu \in BV(\Omega;\R^n)$ is a special function with bounded variation  and we write $\uu \in SBV(\Omega;\R^n)$, if the Cantor part of its distributional derivative $\nabla^c u$ vanishes, i.e.
\[
\nabla  \uu=\nabla^a \uu   \, \mathcal{L}^d+\left(\uu^+-\uu^-\right) \otimes \nu_\uu ~\mathcal{H}^{d-1} \mres J_\uu
\]
where $\nabla^a$ denotes the absolutely continuous part of the  distributional derivative (with respect to    Lebesgue measure $\mathcal{L}^d$) and $J_\uu$ is the jump set of $\uu$ with measure theoretical outer normal vector $\nu_\uu$.
\end{definition}
The following two results will be used to obtain convergences up to the free boundary. We refer the readers to  the monograph of Ambrosio--Fusco--Pallara \cite{MR1857292} for the proofs.
\begin{prop}\label{thmsbv}(Closure of $SBV$) Let $\varphi:[0, \infty) \rightarrow[0, \infty] $ be lower semicontinuous increasing functions and assume that
$\lim _{t \rightarrow \infty} \frac{\varphi(t)}{t}=\infty$.
Let $\Omega \subset \R^n$ be open and bounded, and let $\{\uu_k\} \subset SBV(\Omega)$ such that
\[\sup _k \int_{\Omega} \varphi\left(\left|\nabla^a \uu_k\right|\right) \,d x+\sup _k \int_{J_{\uu_k}}  \left|\uu_k^{+}-\uu_k^{-}\right| \, d \mathcal{H}^{d-1}<\infty.\label{sbvthm4.4}\]
where $\nabla^a \uu_k$ is the absolute continuous part of the distributional  gradient $\nabla \uu_k$, and $(\uu_k^{+},\uu_k^{-})$ are the approximate one-sided limits on the jump set 
$J_{\uu_k}$. If $\{\uu_k\}$ weakly-star converges in $BV(\Omega)$ to $\uu$, then the following statements hold
\begin{itemize}
\item  $\uu \in SBV(\Omega)$.
\item $\nabla^a \uu_k$ weakly converge to $\nabla^a \uu$ in $L^{1}(\Omega)$.
\item The jump part of the gradient  $\nabla^{j} \uu_k$ weakly-star converge to $\nabla^j \uu$ in $\Omega$. 
\item For any convex function $\varphi$, there holds  
\begin{align}
\int_{\Omega} \varphi(|\nabla^a \uu|) \,d x \leq   \liminf _{k \rightarrow \infty} \int_{\Omega} \varphi\left(\left|\nabla^a \uu_k\right|\right)\, d x.
\end{align}\label{LSC sbv}
\end{itemize}

 \end{prop}
\begin{prop}\label{AFP2}
(Compactness of $SBV$) Let $\varphi, \Omega$ be  as in Proposition \ref{thmsbv}. Let $\{\uu_k\}\subset SBV(\Omega)$ be satisfying \eqref{sbvthm4.4} and assume, in addition, that $\left\|\uu_k\right\|_{L^{\infty}(\Omega)}$ is uniformly bounded in $k$. Then, there exists a subsequence
\[\uu_k\xrightarrow{k\to\infty}\uu \in SBV(\Omega) \text{ weakly star in $BV(\O)$}. \] 
\end{prop}

\subsection{Minimal  connections}
We shall  briefly describe some basic properties of minimal orbits. 
For any $\pp^\pm\in \mm_\pm$, we define their minimal connection
\begin{align}\label{der mc3}
 \cc_F(\pp^+,\pp^-)&:= \inf\left\{\int_{\R}  \frac 12 |\bgamma '(t)|^2+ F (\bgamma (t))\, dt \,\Big| \,  \bgamma \in H^1_{loc}(\R;\R^n ),\bgamma (\pm \infty)=\pp^\pm\in \mm_\pm\right\}.
\end{align}

% \begin{lemma}
% The function $ \cc_F(\pp^+,\pp^-): \mm_+\times \mm_-\mapsto \R^+$ is Lipschitz continuous. Moreover, 
%Let $c_F$ be the number defined by \eqref{linpanwang cf equ}, then  
%\[c_F=\inf_{\pp^\pm\in \mm_\pm} \cc_F(\pp^+,\pp^-),\label{def cf}\]
%\end{lemma}

We also introduce     the centralized  potential, which is the even function
\[\tilde{F}(\lambda):= \begin{cases}f\left(\left(\tfrac {\dist_{\mm}}2 +\lambda\right)^2 \right) & \text { if } \lambda \leq 0, \\ f\left(\left(\tfrac {\dist_{\mm}}2 -\lambda\right)^2 \right) & \text { if } \lambda \geq 0,\end{cases}\label{centralized  potential}\]
and the associated scalar-valued minimal connection problem
\[c_{\tilde{F}}:=\min \left\{\int_{\R }\left(\frac 12 \left|\gamma'(s)\right|^2 +\tilde{F}(\gamma(s))\right) d t~\Big|~ \gamma \in H^1_{loc}(\R ), \gamma(\pm \infty)=\pm \tfrac {\dist_{\mm}}2 \right\}.\label{optimial mini}\] 
The following result was obtained in \cite{MR985992,MR930124}.
\begin{lemma}
 It holds that
\[c_{\tilde{F}}:=2 \int_{0}^{\tfrac{\dist_{\mm}}2} \sqrt{2\widetilde{F}(\lambda)} d \lambda \quad =c_F\(=2 \int_{0}^{\tfrac{\dist_{\mm}}2} \sqrt{2f(\lambda^2)} d \lambda\).\label{linpanwang 2.2}\]
Moreover, there exists a minimizer $\alpha(s)$  
of \eqref{optimial mini}  that   satisfies
\begin{subequations}\label{optimal profile alpha}
\begin{align}
&\alpha(s) \in C^{\infty}\left(\R ; (-\tfrac{\dist_{\mm}}2, \tfrac{\dist_{\mm}}2 )\right) \text{ is odd and strictly  increasing in }\R.\label{odd increase}\\
&-\alpha^{\prime \prime}(s)+  \tilde{F}'(\alpha(s))=0, \quad s\in \R ; \quad \alpha(\pm \infty)=\pm \tfrac {\dist_{\mm}}2 .\label{travelling wave}\\
&\alpha'(s)=\sqrt{2\tilde{F}(\alpha(s))}, \quad \forall ~ s \in \R.\label{alpha'=}\\
&\left|\alpha'(s)\right|+\left|\alpha(s)\mp\tfrac {\dist_{\mm}}2 \right| \leq C  e^{-C |s|} \quad \text { as } s \rightarrow\pm\infty.\label{exp alpha}
\end{align}
\end{subequations}
\end{lemma}

We need an equivalent  condition to the minimal pair one stated at \eqref{thm minimal pair}.
To this end, we introduce
\begin{subequations}\label{def Cap M}
\begin{align} 
&M^{+}:=\left\{\pp^{+} \in \mm_+\mid \exists \,\pp^{-} \in \mm_- \text { s.t. }\left|\pp^{+}-\pp^{-}\right|=\dist_{\mm}\right\}, \\
&M^{-}:=\left\{\pp^{-} \in \mm_-\mid  \exists \, \pp^{+} \in \mm_+ \text { s.t. }\left|\pp^{+}-\pp^{-}\right|=\dist_{\mm}\right\}.
\end{align}
\end{subequations}
\begin{lemma}{\cite[Theorem 2.1]{Lin2012a}}\label{lemma mc}
The function $\cc_F(\cdot,\cdot)$ defined by \eqref{der mc3} satisfies 
\[\cc_F(\pp^+,\pp^-)= c_F\qquad \forall\pp^\pm\in \mm_\pm.\label{cf small}\]
 Moreover, we have the following  equivalence for a pair $\pp^\pm\in \mm_\pm$:
\begin{subequations}\label{cf equivalence}
\begin{align}
\bullet \qquad& \cc_F(\pp^+,\pp^-)\text{ is attained by a minimal  orbit }  \bgamma.\label{rigidity cf} \\
\bullet \qquad& \pp^\pm\in M^\pm,\quad  |\pp^+-\pp^-|=\dist_{\mm}.\label{rigidity cf final}
\end{align}
\end{subequations}
Furthermore, assuming one of the above two conditions hold, 
  the corresponding minimal  orbit $\bgamma$ attaining the minimum  of $\cc_F(\pp^+,\pp^-)$  is  
\[
\bgamma(t)=\frac{\pp^++\pp^-}{2}+\alpha(t) \frac{\pp^+-\pp^-}{\dist_{\mm}}, \quad t \in \R ,\label{straightline}
\]
where $\alpha \in H^1_{loc}(\R )$  is a solution to  \eqref{optimial mini} (and equivalently to \eqref{travelling wave}).
\end{lemma}
Seemly more complicated,  the condition  \eqref{rigidity cf} is more compatible with  the variational structure of the functional \eqref{GL energy} than \eqref{rigidity cf final}.  

\subsection{Quasi-distance function}
We   define     the  distance from $\uu$ to $\mm$ according to its relative distance to its two components:
 \[\dd_{\mm}(\uu)=\left\{
\begin{split}
 \dd_{\mm_+}(\uu)&\qquad  \text{ when }   \dd_{\mm_+}(\uu)\leq \tfrac {\dist_{\mm}}2 ,\\
\dd_{\mm_-}(\uu)&\qquad  \text{ when }   \dd_{\mm_-}(\uu)\leq \tfrac {\dist_{\mm}}2.
\end{split}
\right.\label{dN global}\]
In general such a function is not $C^1$ unless $\uu$ is close to but not in $\mm$ (cf. \cite[Section 4.4]{MR2977424}). We assume that the nearest point projection $P_{\mm}:B_{\delta_0}(\mm)\mapsto \mm$ is smooth. Using this, we have 
\[\dd_{\mm}(\uu):=|\uu-P_{\mm}(\uu)|,\quad \forall ~\uu\in B_{\delta_0}(\mm).\]
Such a function is $C^1$ in $B_{\delta_0}(\mm)\backslash \mm$,  and this motivates the following unit vector field:
\begin{align}
  \bnu_{\mm} (\uu):=&\begin{cases}
\p \dd_{\mm}(\uu)=\frac{\uu-P_{\mm}(\uu)}{|\uu-P_{\mm}(\uu)|},\quad &\forall ~\uu\in B_{\delta_0}(\mm)\backslash \mm,\\
0
\quad &\forall~ \uu\in   \mm.
\end{cases}\label{normal of m}
\end{align}
Note that such a normal vector field is in general  not continuous up to $\mm$ unless $\mm_\pm$ are hypersurfaces of $\R^n$.

With these preparations,  we introduce  the quasi-distance function
\[\label{quasidistance}
  \dd_F( \uu):= 
  \left\{
\begin{split}
  \int_0^{\dd_\mm(\uu)}\sqrt{2f(\lambda^2)}\, d\lambda\quad &\text{ if }   \dd_{\mm_-}(\uu)\leq \tfrac{\dist_\mm}2,\\
   \tfrac 12 c_F \quad &\text{ if }  \dd_{\mm}(\uu)> \tfrac{\dist_\mm}2,\\
   c_F-\int_0^{\dd_\mm(\uu)}\sqrt{2f(\lambda^2)}\, d\lambda\quad & \text{ if } \dd_{\mm_+}(\uu)\leq \tfrac{\dist_\mm}2,
 \end{split}
\right.
\]
where $c_F=2 \int_{0}^{\frac {\dist_{\mm}}2} \sqrt{2f(\lambda^2)} d \lambda$ is the surface tension coefficient (cf. \eqref{linpanwang cf equ}). 
Such a  function    is a modification of the one used in   \cite{MR930124,MR985992}. We list   some of  its important properties  here.
 \begin{lemma}\label{lemma quasidis}
The function $\dd_F:\R^n\mapsto [0,c_F]$  is   Lipschitz continuous  in $\R^n$,  and  satisfies 
 \begin{align}
 \label{eq:2.7}|\p   \dd_F( \uu)| &\leq  \sqrt{2F (\uu)}\quad a.e.~\uu\in \R^n,\\
\label{eq:1.6}
   \dd_F(\uu)&=\left\{
   \begin{array}{rl}
   0\qquad\text{if and only if}&~\uu\in \mm_-,\\
    c_F\qquad \text{if and only if}&~\uu \in  \mm_+.
   \end{array}
   \right.\\
   \dd_F &\in C^1(B_{\delta_0}(\mm)),\text{ and } \p \dd_F|_{B_{\delta_0}(\mm)}(\uu)=0~ \text{if and only if }~ \uu\in \mm.\label{df is c1}
   \end{align}
Moreover,     $\frac{\p \dd_F(\uu)}{| \p \dd_F(\uu)|}$ is a continuous unit vector field in $B_{\delta_0}(\mm)\backslash \mm$  so that 
\[\frac{\p \dd_F(\uu)}{| \p \dd_F(\uu)|}= \bnu_\mm(\uu)\qquad \forall \uu \in B_{\delta_0}(\mm)\backslash \mm.\label{normalize pdf}\]

  \end{lemma}

\begin{proof}
 It is   obvious that $\dd_F$ is continuous in $\R^n$, and is Lipschitz in each subdomain where it is defined. It suffices to check the Lipschitz condition across  adjacent regions. If $\dd_{\mm}(\uu_1)> \tfrac{\dist_\mm}2$ and $\dd_{\mm_-} (\uu_2)\leq  \tfrac{\dist_\mm}2 $, then  
 \begin{align*}
0&\leq \dd_F(\uu_1)-\dd_F(\uu_2)=\int_{\dd_{\mm_-}(\uu_2)}^{\frac {\dist_\mm}2}\sqrt{2f(\lambda^2)}\, d\lambda\\
&\leq \int_{\dd_{\mm_-}(\uu_2)}^{\dd_{\mm_-}(\uu_1)}\sqrt{2f(\lambda^2)}\, d\lambda\leq C |\dd_{\mm_-}(\uu_1)-\dd_{\mm_-}(\uu_2)|\leq C|\uu_2-\uu_1|.
\end{align*}
  Other cases can be treated in a similar way. The inequality \eqref{eq:2.7} and the formula \eqref{eq:1.6} follows directly from the definition \eqref{quasidistance}.

To show $\dd_F\in C^1(B_{\delta_0}(\mm))$, it suffices to write it as a function of $\dd_\mm^2$ which is smooth up to $\mm$.  Indeed, by  \eqref{bulk2} one can verify that   
 \[h(s)=\int_0^{s^{1/2}}\sqrt{2f(\lambda^2)}\, d\lambda\in C^1[0,\infty).\label{s12 func}\] 
 As a result,   $\dd_F(\uu)=h( \dd_\mm^2)$ in $B_{\delta_0}(\mm_-)$, and a  similar argument applies to $B_{\delta_0}(\mm_+)$. The rest assertions are due to $\dd_\mm\in C^1(B_{\delta_0}(\mm)\backslash \mm)$.

\end{proof}
%%%%%%%%%%%%%%%%%%%%%%%%%%%%%%%%%%%%%%%%%%%%%%%%%%%%%%%%%%%%%%%%%%%%%%%%%%%%%%%%%%%%%%%%%%%%%%%%%%%%%%%%%%%%

\subsection{Geometry of   interfaces}\label{subsection geo}
Under a local parametrization $\bphi_t(s):U\to \Sigma_t$ on  an open set $U\subset \R^{d-1}$, the MCF equation  writes
\[\p_t \bphi_t(s)=\kappa(\bphi_t(s),t) \nn(s,t), \label{csf}\]
where $\kappa$ is the mean curvature and $\nn$ is   the inward normal.
For   $\delta>0$,  the $\delta$-neighborhood of $\Sigma_t$ is  the open set \begin{equation}
B_\delta(\Sigma_t):=  \{x\in\Omega:  | \dd_\Sigma(x,t)|<\delta\}.
\end{equation}
We shall choose the $\delta_0$ (first appeared in \eqref{bulk2}) small enough so that   the nearest point projection $$P_{\Sigma}(\cdot,t): B_{4\delta_0}(\Sigma_t) \mapsto \Sigma_t$$ is smooth for any $t\in [0,T]$, and   the interface \eqref{interface} keeps  at least $4\delta_0$ distance  away  from the boundary of the domain $\p\O$.  
 Analytically we have $$P_\Sigma(x,t)  =x-\nabla \dd_\Sigma (x,t) \dd_\Sigma (x,t).$$ So for  each fixed $t\in [0,T]$,  any point $x\in B_{4\delta_0}(\Sigma_t)$ corresponds to a unique pair $(r,s)$ with    $r=\dd_\Sigma (x,t)$ and $s\in U$, and thus   the identity
$$\dd_\Sigma (\bphi_t(s)+r\nn(s,t), t)\equiv r$$ holds with   independent variables  $(r,s,t)$.  
  Differentiating this identity with respect to $r$ and $t$   leads to  the following identities: 
  \[\nabla \dd_\Sigma (x,t)= \nn(s,t),\qquad -\p_t \dd_\Sigma (x,t)=\p_t \bphi_t(s)\cdot\nn(s,t)=: V(s,t).\label{velocity}\]
    These   formulas   extend  the  inward normal vector $\nn$ and the normal velocity $V$  of $\Sigma_t$ to  $B_{4\delta_0}(\Sigma_t)$.

Now we come to the definition of   $\bxi$    in the  modulated energy  $E_\e  [\uu_\e  | \Sigma](t)$  \eqref{entropy}. This is done by  extending  the inward normal vector field $\nn$ through
\[\bxi (x,t)=\phi \( \frac{\dd_\Sigma(x,t)}{\delta_0}\)\nabla \dd_\Sigma(x,t).\label{def:xi}\]
Here   $\phi(x)\geq 0$ is  an     even, smooth function on $\R$ that    decreases for  $x\in [0,1]$, and satisfies 
\begin{equation}\begin{cases}
\phi(x)>0~&\text{ for }~|x|< 1,  \\
\phi(x)=0~&\text{ for }~|x|\geq 1, \\
1-4 x^2\leq \phi(x)\leq 1-\frac 12 x^2~&\text{ for }~|x|\leq  1/2.
\end{cases}\label{phi func control}
\end{equation}

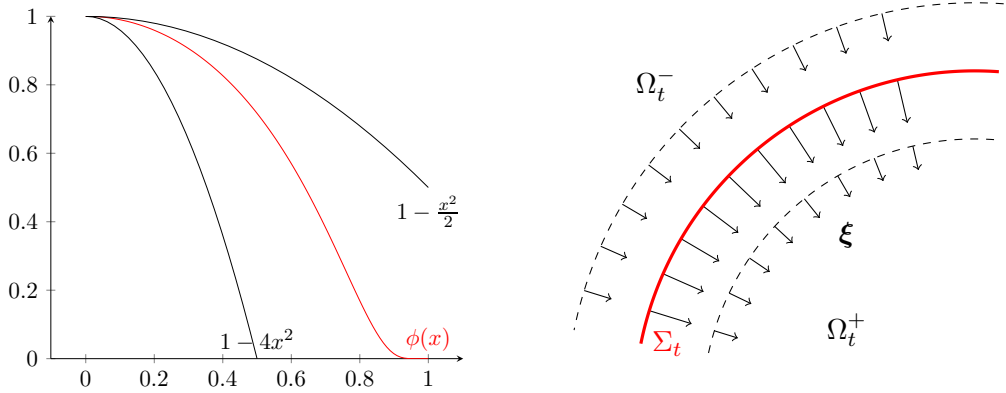
\picdis{\begin{tikzpicture}[scale = 0.8]
\begin{axis}[axis equal,axis lines = left,
 %   title=graph of $\phi(x)$
]
  %  \legend{$\phi(x)$}
\addplot[domain=0: 0.9999,color=red,samples=100]{exp(x^2/(x^2-1))} node[above] {$\phi(x)$};
\addplot[domain=0:0.5,color=black]{1-4*x^2} node[above] {$1-4x^2$};
\addplot[domain=0:1,color=black]{1-0.5*x^2} node[below] {$1-\frac{x^2}2$};
\end{axis}

\end{tikzpicture}
\qquad 
\begin{tikzpicture}[scale = 1]
\begin{axis}[axis equal,
 axis lines=none,
  %axis line style={->},
 %tick style={color=black},
 xtick=\empty,
 ytick=\empty,
% title=extend inner normal $\bxi $
]

  \draw (30, 100) node   {$\O_t^-$};
    
  \addplot[samples=8, domain=1.8:pi-0.3, 
	variable=\t,
	quiver={
		u={-cos(deg(t))},
		v={-sin(deg(t))},
		scale arrows=0.09},
		->,black]
	({cos(deg(t))}, {sin(deg(t))});
	  \addplot[samples=10, domain=1.8:pi-0.3, 
	variable=\t,
	quiver={
		u={-1.5*cos(deg(t))},
		v={-1.5*sin(deg(t))},
		scale arrows=0.07},
		->,black]
	({1.5*cos(deg(t))}, {1.5*sin(deg(t))}); 
	  \addplot[samples=10, domain=1.8:pi-0.3, 
	variable=\t,
	quiver={
		u={-1.25*cos(deg(t))},
		v={-1.25*sin(deg(t))},
		scale arrows=0.13},
		->,black]
	({1.25*cos(deg(t))}, {1.25*sin(deg(t))}) ;
	\addplot[samples=100, domain=1.5:pi-0.2,dashed] 
	({1.5*cos(deg(x))}, {1.5*sin(deg(x))});
  \addplot[samples=100, domain=1.5:pi-0.2, dashed,black] 
	({cos(deg(x))}, {sin(deg(x))}) ;
	\addplot[samples=100, domain=1.5:pi-0.2, very thick,red] 
	({1.25*cos(deg(x))}, {1.25*sin(deg(x))}) node[right]{$\Sigma_t$};
	  \draw (100, 45) node   {$\bxi $};
  \draw (100, 10) node   {$\O_t^+$};
\end{axis}
\end{tikzpicture}}

To fulfill these requirements, we can simply choose 
\[\phi(x)=e^{\frac 1{x^2-1}+1}~\text{for}~|x|< 1~\text{and}~\phi(x)=0~\text{for}~|x|\geq 1.\label{phi func}\]  
 We  also need to extend the  curvature   of $\Sigma_t$ to the domain  $B_{4\delta_0}(\Sigma_t)$. To this end, choose a cut-off function 
  \[\eta_0\in C_c^\infty(B_{2\delta_0}(\Sigma_t))~\text{ with }\eta_0=1\text{  in }B_{\delta_0}(\Sigma_t),\label{cut-off eta delta}\] 
and we  define 
\[ \HH(x,t)=\kappa \nabla \dd_\Sigma (x,t)  \quad\text{with}\quad \kappa(x,t)=-\Delta \dd_\Sigma  (P_\Sigma(x,t), t )\eta_0(x,t).\label{def:H}\]
By \eqref{cut-off eta delta}, $\HH$ is extended constantly in the  normal direction. So we have
 \begin{align}
 (\nn\cdot\nabla )\HH&=0\text{ and } (\bxi \cdot\nabla )\HH=0\quad \forall t\in [0,T],~ x\in B_{2\delta_0}(\Sigma_t).\label{normal H}\\
 \label{bc n and H}
   \bxi &=0 ~\text{and}~\HH=0\quad \forall t\in [0,T],~ x\in \p\O.
  \end{align}
 We end this part by  the following identities which will be employed to prove the modulated energy   inequalities:
  \begin{subequations}\label{xi der}
  \begin{align}
     \nabla\cdot \bxi +\HH \cdot \bxi &= O(\dd_\Sigma ),\label{div xi H}\\
  \p_t \dd_\Sigma (x,t) +(\HH( x,t)\cdot\nabla)  \dd_\Sigma (x,t) &=0\quad \text{in}~B_{\delta_0}(\Sigma_t),\label{mcf}\\
\p_t \bxi +\left(\HH \cdot \nabla\right) \bxi +\left(\nabla \HH\right)^{\mathsf{T}} \bxi &=0\quad \text{in}~B_{\delta_0}(\Sigma_t),\label{xi der1} \\
\p_t |\bxi |^2 +\left(\HH \cdot \nabla\right)|\bxi |^2 &=0\quad \text{in}~B_{\delta_0}(\Sigma_t),\label{xi der2}
\end{align}
  \end{subequations}
  where $\nabla \HH:=\{\p_j H_i\}_{1\leq i, j\leq d}$ is a matrix with $i$ being the row index. 
  \begin{proof}[Proof of \eqref{xi der}]
  Recalling  \eqref{def:xi},  $\phi_0(\tau):=\phi (\frac \tau{\delta_0})$   is an even function. So it follows from   $\phi_0'(0)=0$ and Taylor's expansion in $\dd_\Sigma$   that 
\begin{align*}
\nabla\cdot \bxi &=|\nabla \dd_\Sigma|^2 \phi_0'(\dd_\Sigma)+\phi_0(\dd_\Sigma)\Delta \dd_\Sigma(x,t)
\\&=O(\dd_\Sigma) +\phi_0 (\dd_\Sigma)\Delta \dd_\Sigma(P_I(x,t),t),
\end{align*}
and this together with \eqref{def:H} leads to   \eqref{div xi H}.    
Using  \eqref{velocity} and \eqref{def:H}, we can write \eqref{csf} as the  transport equation \eqref{mcf}:
$$-\p_t \dd_\Sigma=\p_t \bphi_t(s)\cdot\nn(s,t)=\kappa(\bphi_t(s),t)=\HH\cdot\nabla\dd_\Sigma\quad \text{in}~B_{\delta_0}(\Sigma_t).$$
This equation implies  that the following two  identities  hold  in $B_{\delta_0}(\Sigma_t)$:
\begin{align*}
\p_t \nabla \dd_\Sigma+(\HH \cdot \nabla) \nabla \dd_\Sigma +(\nabla \HH )^{\mathsf{T}} \nabla \dd_\Sigma=0,\\
\p_t \phi_0(\dd_\Sigma )+ (\HH  \cdot\nabla) \phi_0(\dd_\Sigma)=0.
\end{align*}
 These two equations together  imply \eqref{xi der1}. Finally \eqref{xi der2} is a consequence of \eqref{mcf}.
  \end{proof}

\subsection{Modulated energy  method}\label{sec entropy}
As the gradient flow of the Ginzburg--Landau energy \eqref{GL energy}, the system
\eqref{Ginzburg-Landau} has the following energy dissipation law
\begin{equation}\label{dissipation}
A_\e (\uu_\e  (\cdot,T))+  \int_0^T \int_\O \e  |\p_t \uu_\e  |^2 \,d x \,d t=A_\e (\uu_\e (\cdot,0)),~\text{for all}~ T> 0.
\end{equation}
For initial datum undergoing a phase transition across    the initial interface $\Sigma_0$, due to   concentrations  of $\nabla \uu_\e  $ on   $\Sigma_t$, the  dissipation law \eqref{dissipation} is not sufficient to derive quantitative  convergences of $\uu_\e$, not even away from $\Sigma_t$. Following a recent work of  Fisher--Laux-Simon  \cite{fischer2020convergence} we shall derive   an inequality which modulates  the concentrations  and   leads to    compactness of $\{\uu_\e\}$ in Sobolev spaces.

  We shall start by discussing  the differentiability of $\psi_\e(x,t)=\dd_F \circ \uu_\e (x,t)$ (cf. \eqref{psi}).
It follows from  Lemma \ref{lemma quasidis} that $\dd_F(\cdot)$ is  a  Lipschitz function in $\R^n$ with $\dd_F(0)=0$ under the     assumption that $0\in \mm^-$.  Following Laux--Simon \cite{MR3847750}, for  every $ (x,t)\in \O\times  [0,T]$, we consider   the restriction of    $\dd_F(\uu_\e(x,t))$ to  the affine space
\begin{align*}
&T^{\uu_\e}_{x,t}:=\uu_\e(x,t)+ {\rm span}\{\p_0   \uu_\e(x,t) ,\cdots,  \p_d\uu_\e(x,t)\},
\end{align*}
 denoted  by $\dd_F|_{T^{\uu_\e}_{x,t}}$. By the generalized  chain rule of Ambrosio--Dal Maso \cite{MR969514}, $\dd_F|_{T^{\uu_\e}_{x,t}}$
is differentiable at $\uu_\e(x,t)$. Now we denote the orthogonal   projection from $\R^n$ to the subspace $T^{\uu_\e}_{x,t}-\uu_\e(x,t)$  by $\Pi_{x,t}^{\uu_\e}$,  and define the generalized differential by 
\[\p \dd_F(\uu_\e):=\p \(\dd_F|_{T^{\uu_\e}_{x,t}}\)\Big|_{\uu_\e(x,t)}\circ \Pi^{\uu_\e}_{x,t}.\label{def of linear map in chain}\]   
Then we have for $0\leq i\leq d$ that 
\[\p \dd_F(\uu_\e)\cdot \p_i \uu_\e=\p \(\dd_F|_{T^{\uu_\e}_{x,t}}\)\Big|_{\uu_\e(x,t)}\cdot  \p_i \uu_\e=\p_i  (\dd_F\circ \uu_\e)\]
where the second  equality is due to the directional derivative at $\uu_\e(x,t)$ pointing to $\p_i \uu_\e(x,t)$.
This proves   the generalized chain rule  
  \begin{align}
  \label{ADM chain rule}
  \p_i\psi_\e (x,t)   =    \p_i \uu_\e   (x,t)\cdot\p\dd_F\(\uu_\e   (x,t)\)\quad \text{for } 0\leq i\leq d\text{ and a.e. } (x,t).  \end{align}
  Moreover,  the  point-wise differential inequality \eqref{eq:2.7} is valid for the  generalized differential \eqref{def of linear map in chain}:
   \[\label{eq:2.7global}|\p   \dd_F( \uu_\e)|\leq  \sqrt{2F (\uu_\e)}. \]
   Note that the above inequality, rather than its classical version  \eqref{eq:2.7},  will be used   to  prove the modulated energy inequalities.  
   
To proceed, we define  the phase-field analogues of   the normal vector and  the mean curvature  vector  respectively  by
\begin{subequations}
\begin{align}
 \nn_\e (x,t)&:=\begin{cases}
 \frac{\nabla \psi_\e }{|\nabla \psi_\e|}(x,t)&\text{ if } \nabla \psi_\e (x,t)\neq 0,\\
0& \text{ otherwise}.
 \end{cases}
\label{normal diff}\\
\HH_\e (x,t)&:=\begin{cases}
-\left(\e  \Delta \uu_\e  -\frac{1}{\e }\p  F(\uu_\e  ) \right)\cdot\frac{\nabla \uu_\e  }{\left|\nabla \uu_\e  \right|} &\text{ if } \nabla \uu_\e(x,t)\neq 0,\\
0&\text{ otherwise}.
\end{cases}
 \label{mean curvature app}
\end{align}
\end{subequations}
 Note that in \eqref{mean curvature app}, the inner product is made with  the column  vectors of $\nabla \uu_\e=(\p_1 \uu_\e,\cdots,\p_d \uu_\e)$.

Using \eqref{normal of m} and \eqref{normalize pdf} we define the  linear projection  
\[ \label{projection1}
\Pi_{\uu_\e  }  \p_i \uu_\e  :=
\left\{
\begin{split}
 \p_i \uu_\e \cdot \bnu_\mm(\uu_\e)  \bnu_\mm(\uu_\e)&~\text{ if }~  \uu_\e\in B_{\delta_0}(\mm),\\
\(\p_i \uu_\e  \cdot\frac{\p  \dd_F   (\uu_\e  )  } {|\p  \dd_F   (\uu_\e  )|}\) \frac{\p\dd_F(\uu_\e)}{|\p  \dd_F   (\uu_\e  )|} &~\text{ if }~ \uu_\e\notin B_{\delta_0}(\mm)\text{ and } \p\dd_F(\uu_\e)  \neq 0,\\
0 &~ \text{ if }~ \uu_\e\notin B_{\delta_0}(\mm) \text{ and } \p\dd_F(\uu_\e)  = 0
\end{split}
\right.
\]
where $\p \dd_F$ is interpreted as the generalized differential \eqref{def of linear map in chain} in case $\dd_F$ is not classically  differentiable at $\uu_\e$. 
\begin{lemma}
The following two identities hold:
\begin{align}
\label{projectionnorm}
 |\nabla \psi_\e | &= |\Pi_{\uu_\e  } \nabla \uu_\e  | |\p  \dd_F   (\uu_\e  )|  &\text{ a.e. } (x,t),\\
  \label{projection}
\Pi_{\uu_\e  } \nabla \uu_\e  &=\frac{|\nabla\psi_\e |} {|\p  \dd_F   (\uu_\e  )|^2}\p  \dd_F   (\uu_\e  )\otimes \nn_\e    &\text{ if } \quad\p \dd_F   (\uu_\e  )\neq 0,
\end{align}
where in \eqref{projection} the projection applies to each column vector  of  $\nabla \uu_\e$.
\end{lemma}
\begin{proof}

Concerning \eqref{projectionnorm}, we distinguish two cases:

(a) When $\uu_\e\in B_{\delta_0}(\mm)$, then according to \eqref{df is c1}, $\dd_F$ is $C^1$ and the generalized differential \eqref{def of linear map in chain} coincide with the classical one. 
Thus on the set $\{x\mid \uu_\e(x,t)\in B_{\delta_0}(\mm)\backslash \mm\}$, we have 
\begin{align}
  \p_i\psi_\e   \overset{\eqref{ADM chain rule}}=    \p_i \uu_\e   \cdot \frac{\p\dd_F\(\uu_\e \)}{|\p\dd_F\(\uu_\e \)|}|\p\dd_F\(\uu_\e\)|    \overset{\eqref{normalize pdf}}= \p_i \uu_\e   \cdot \bnu_\mm(\uu_\e)|\p\dd_F\(\uu_\e\)|.
  \end{align}
This together with  the first case of \eqref{projection1} leads to \eqref{projectionnorm}. On the set $\{x\mid \uu_\e(x,t)\in  \mm\}$,  we have   from \eqref{df is c1} that $\p \dd_F(\uu_\e)=0$.  By \eqref{ADM chain rule}, both sides of \eqref{projectionnorm} vanishes.

(b) When $\uu_\e\notin B_{\delta_0}(\mm)$: if   $\p\dd_F(\uu_\e)  \neq 0$, then owning to \eqref{ADM chain rule} and the second  case in \eqref{projection1}, we obtain   \eqref{projectionnorm}; if    $\p\dd_F(\uu_\e)  = 0$, then by \eqref{ADM chain rule}, we have $\nabla \psi_\e=0$ too. This finishes the proof of \eqref{projectionnorm}.

Now we turn to the proof of  the formula    \eqref{projection} under the assumption  that $\p \dd_F   (\uu_\e  )\neq 0$. It holds when $\nabla\psi_\e=0$ because   \eqref{projectionnorm} then implies $\Pi_{\uu_\e  } \nabla \uu_\e=0$. When $\nabla\psi_\e\neq 0$, then 
\begin{align}\label{chainlast1}
\frac{|\nabla\psi_\e |} {|\p  \dd_F   (\uu_\e  )|^2}\p  \dd_F   (\uu_\e  )\otimes \nn_\e\overset{
 \eqref{normal diff}}=\frac{\p  \dd_F   (\uu_\e  )  } {|\p  \dd_F   (\uu_\e  )|^2}\otimes \nabla\psi_\e \nonumber\\\overset{\eqref{ADM chain rule}}=  \frac{\p  \dd_F   (\uu_\e  )  } {|\p  \dd_F   (\uu_\e  )|^2}\otimes \(\nabla \uu_\e\cdot\p\dd_F(\uu_\e)\).
\end{align}

We distinguish two cases:

(1) When $\uu_\e\notin B_{\delta_0}(\mm)$, the last term of \eqref{chainlast1}  is exactly the second case defining   $\Pi_{\uu_\e  } \nabla \uu_\e$ \eqref{projection1}. 

(2) When $\uu_\e\in  B_{\delta_0}(\mm)$, by \eqref{df is c1} we have $\uu_\e\notin \mm$ (under the assumption $\p \dd_F   (\uu_\e  )\neq 0$). So  by \eqref{normalize pdf}, the last term of \eqref{chainlast1} simplifies to 
$\big(\nabla \uu_\e \cdot \bnu_\mm(\uu_\e)\big) \otimes \bnu_\mm(\uu_\e)$. This coincides  with the first  case of   \eqref{projection1} defining $\Pi_{\uu_\e  }  \nabla \uu_\e$.
This finishes the proof of  \eqref{projection}.
\end{proof}

	As we shall not integrate the time variable $t$ throughout this section,   we shall abbreviate the spatial integration $\int_\O$ by $\int$ and sometimes we omit the $\,dx$.
The following lemma gives various coercivity estimates of  $E_\e  [\uu_\e  | \Sigma]$   \eqref{entropy}. It   was due to \cite{MR4284534}, generalizing the one by \cite{fischer2020convergence} to   vectorial cases. We present the proof for the convenience of the readers.
\begin{lemma}\label{lemma:energy bound}
There exists a universal constant $C>0$ which is independent of $t\in [0,T]$ and $\e $  such that the following estimates hold  for  every $t\in [0,T]$:
\begin{subequations} \label{energy bound}
\begin{align}
 \int \(\frac{\e }{2} \left|\nabla \uu_\e  \right|^2+\frac{1}{\e } F  (\uu_\e  )-|\nabla \psi_\e | \) \, d x & \leq E_\e  [ \uu_\e   | \Sigma ] , \label{energy bound-1}\\
 \e  \int   \left|\nabla \uu_\e  -\Pi_{\uu_\e  }\nabla \uu_\e  \right|^2    \, d x & \leq 2   E_\e  [ \uu_\e   | \Sigma ] ,\label{energy bound0}\\
  \int\left(\sqrt{\e }\left|\Pi_{\uu_\e  }\nabla \uu_\e  \right|-\frac1{\sqrt{\e }} \left|\p   \dd_F   (\uu_\e  )\right| \right)^2\, d x & \leq  2    E_\e  [ \uu_\e   | \Sigma ] ,\label{energy bound2}\\
%&\int \left(\sqrt{\e }\left|\nabla \uu_\e  \right|-\frac{1}{\sqrt{\e }}\left|\p   d_F   (\uu_\e  )\right| \right)^2 d x\leq C E_\e  [ \uu_\e   | \Sigma ] ,\label{energy bound2tilde}\\
 \int\( {\frac{\e }{2}}\left| \nabla \uu_\e  \right|^2 +\frac{1}{\e } F  (\uu_\e  )+\left|\nabla \psi_\e\right|\)\left(1-\bxi  \cdot\nn_\e\right) \, d x & \leq C E_\e  [ \uu_\e   | \Sigma ] ,\label{energy bound1}
\\
   \int \(\frac{\e }2 \left|\nabla \uu_\e  \right| ^2 +\frac{1}{\e } F  (\uu_\e  )+|\nabla\psi_\e |\) \min\(\dd_\Sigma^2,1\)\, d x  & \leq C E_\e  [ \uu_\e   | \Sigma ].
\label{energy bound3}
 \end{align}
\end{subequations}
\end{lemma} 
\begin{proof}
Using \eqref{normal diff}, we obtain $\nabla\psi_\e=|\nabla\psi_\e|\nn_\e$. Note also that \eqref{projection1}  implies 
 \[\left|\nabla \uu_\e  -\Pi_{\uu_\e  }\nabla \uu_\e  \right|^2+\left| \Pi_{\uu_\e  }\nabla \uu_\e  \right|^2=\left|\nabla \uu_\e    \right|^2.\label{gougudingli}\] 
 Altogether,   we  can write
 \begin{align*}
 &\frac \e  2\left|  \nabla \uu_\e  \right|^2     +\frac{1}{\e } F  (\uu_\e  )-\bxi\cdot\nabla \psi_\e\nonumber\\
 = &   ~    \frac \e  2\left|  \nabla \uu_\e  \right|^2     +\frac{1}{\e } F  (\uu_\e  )-|\nabla \psi_\e | +   |\nabla\psi_\e | (1-\bxi \cdot\nn_\e )\nonumber\\
= & ~ \frac{\e }2   \left|\nabla \uu_\e  -\Pi_{\uu_\e  }\nabla \uu_\e  \right|^2  +   |\nabla\psi_\e | (1-\bxi \cdot\nn_\e )  \nonumber \\
&~ + \frac \e  2\left| \Pi_{\uu_\e  }\nabla \uu_\e  \right|^2     +\frac{1}{\e } F  (\uu_\e  )-|\nabla \psi_\e |.
\end{align*}
This combined with   \eqref{eq:2.7global} and \eqref{projectionnorm} yields
 \begin{align}
 &\frac \e  2\left|  \nabla \uu_\e  \right|^2     +\frac{1}{\e } F  (\uu_\e  )-\bxi\cdot\nabla \psi_\e\nonumber\\
 \geq  &~  \frac{\e }2   \left|\nabla \uu_\e  -\Pi_{\uu_\e  }\nabla \uu_\e  \right|^2  +   |\nabla\psi_\e | (1-\bxi \cdot\nn_\e )  \nonumber \\
&~ + \frac 12 \left(\sqrt{\e }\left|\Pi_{\uu_\e  }\nabla \uu_\e  \right|-\frac1{\sqrt{\e }} \left|\p   \dd_F   (\uu_\e  )\right| \right)^2.\label{pointwise GL}
\end{align}
This  inequality  implies   
\eqref{energy bound-1}, \eqref{energy bound0} and \eqref{energy bound2}.

 Combining \eqref{energy bound-1} with  
$E_\e  [ \uu_\e   | \Sigma ]\geq   \int\left(1-\bxi  \cdot\nn_\e\right)\left|\nabla \psi_\e\right|$
  and $1-\bxi  \cdot\nn_\e\leq 2$  yields \eqref{energy bound1}.
Finally, by \eqref{phi func control} and $\delta_0\in (0,1)$ we have 
\[1-\bxi  \cdot\nn_\e \geq 1-\phi\(\frac {\dd_\Sigma}{\delta_0}\)  \geq \min \(\frac {\dd_\Sigma^2}{2\delta_0^2}, 1-\phi(\tfrac 1 2)\)\geq  C_{\phi,\delta_0} \min(\dd_\Sigma^2,1).\label{lowerbdcali}\] 
This together with  \eqref{energy bound1} implies  \eqref{energy bound3}.
\end{proof}

The following result was first proved in \cite{fischer2020convergence} in  the scalar case, and was   generalized to a matrix-valued  model   in \cite{MR4284534}. 
We present the proof in  Appendix \ref{appendix} for the convenience of the readers.
\begin{prop}\label{gronwallprop}
	There exists a constant $C=C(\Sigma)$ independent of  $\e$ such that
	\begin{align}
	\frac{d}{d t} E_\e  [ \uu_\e  | \Sigma] &+\frac 1{2\e }\int \(\e ^2 | \p_t \uu_\e   |^2-|\HH_\e |^2\)\,dx+\frac 1{2\e }\int \Big| \HH_\e -\e   |\nabla \uu_\e |\HH \Big|^2\,dx  \nonumber \\
	&+\frac 1{2\e }\int \Big| \e \p_t \uu_\e   -(\nabla\cdot  \bxi )\p \dd_F (\uu_\e )    \Big|^2\,dx  \leq CE_\e  [ \uu_\e  | \Sigma]. \label{gronwall}
	\end{align}
\end{prop}
%%%%%%%%%%%%%%%%%%%%%%%%%%%%%%%%%%
 
%%%%%%%%%%%%%%%%%%%%%%%%%%%%
\section{Estimates  of  level sets}\label{sec level}
The main task of this section is to derive the convergence rate estimate \eqref{volume convergencethm} and use it to obtain fine estimates of the level sets of $\psi_\e$. We start with a corollary of  Proposition \ref{gronwallprop}. 
\begin{lemma}\label{lemma level}
There exists a universal constant $C=C( \Sigma)$ such that   
\begin{subequations}
 \begin{align}
 &\sup_{t\in [0,T]} \int_\O \( \frac{\e}2 \left|\nabla \uu_\e \right| ^2 +\frac1 {\e}{F (\uu_\e )}-\bxi\cdot \nabla \psi_\e  \)\, d x\leq C\e,\label{calibration est2}\\
 \label{energy bound4}
&\sup_{t\in  [0,T]} \int_\O \(  \left|\nabla \uu_\e -\Pi_{\uu_\e }\nabla \uu_\e \right|^2   \)\, dx+ \int_0^T\int_\O \(  \left|\p_t \uu_\e -\Pi_{\uu_\e }\p_t \uu_\e \right|^2   \)\, dxdt\leq C,\\
 &\sup_{t\in [0,T]} A_\e(\uu_\e(\cdot,t)) +   \sup_{t\in [0,T]} \|\nabla\psi_\e(\cdot, t)\|_{L^1(\O) } \leq   C,\label{nablapsiest}\\
&\sup_{t\in [0,T]} \int_\O   \(|\nabla \psi_\e|-\bxi\cdot \nabla \psi_\e  \) \, d x\leq C\e\label{calibration est1}.
 \end{align}
\end{subequations}
 Moreover,  for any fixed $\delta\in (0, \delta_0)$, there holds
 \begin{subequations}\label{est away}
 \begin{align}\label{space der bound local}
\sup_{t\in  [0,T]}\int_{\O^\pm_t\backslash B_\delta(\Sigma_t)}\(\frac 12 |\nabla \uu_\e |^2+\frac1{\e^2}F(\uu_\e )\)\, dx \leq \delta^{-2}C,\\
\int_0^T\int_{\O^\pm_t\backslash B_\delta(\Sigma_t)} |\p_t \uu_\e |^2\, dx dt\leq \delta^{-2}C.\label{time der bound local}
\end{align}
 \end{subequations}
 \end{lemma}
\begin{proof}
To prove \eqref{calibration est2},  we need to show that  the first integral on the left-hand side of \eqref{gronwall} is non-negative so that the Gr\"{o}nwall's lemma can be applied. It follows from \eqref{gronwall} and the assumption \eqref{initial}  that
\begin{align}\label{energy1}
& \sup_{t\in  [0,T]} \frac{1}\e  E_\e  [ \uu_\e  | \Sigma] (t)+\frac 1{\e ^2}\int_0^T\int_\O \Big| \e \p_t \uu_\e   -\p \dd_F (\uu_\e )   (\nabla\cdot\bxi) \Big|^2\, dxdt \nonumber  \\
& +\frac 1{\e ^2}\int_0^T\int_\O \(\e ^2 | \p_t \uu_\e   |^2-|\HH_\e |^2+  \Big| \HH_\e -\e  \HH |\nabla \uu_\e |\Big|^2\)\, dxdt\nonumber\\
&\qquad \leq \frac 1\e  e^{(1+T)C( \Sigma)} E_\e [\uu_\e  | \Sigma](0) \leq C_1 e^{(1+T)C( \Sigma)}.  \end{align}
Now we show   the third term  on the left-hand side of  \eqref{energy1} has a  non-negative integrand.
By \eqref{Ginzburg-Landau} and   \eqref{mean curvature app} we have  $\HH_\e =-\e  \p_t \uu_\e \cdot\frac{\nabla \uu_\e }{|\nabla \uu_\e |}$. Using this formula, we can   expand the integrand of the third term  on the LHS of \eqref{energy1}   and then apply  the Cauchy-Schwarz inequality to obtain
\begin{align*}
&~\e ^2 | \p_t \uu_\e   |^2-|\HH_\e |^2+  \Big| \HH_\e -\e  \HH |\nabla \uu_\e |\Big|^2\\
=&~\e ^2 | \p_t \uu_\e   |^2+\e ^2 |\HH|^2 |\nabla \uu_\e |^2+2\e ^2 (\HH\cdot \nabla) \uu_\e \cdot\p_t \uu_\e \\
\geq  &~\e ^2|\p_t \uu_\e +(\HH \cdot\nabla) \uu_\e |^2.
\end{align*}
This   together with \eqref{energy1} implies  
\begin{equation}\label{time est1}
\int_0^T\int_\Omega |\p_t \uu_\e +(\HH \cdot\nabla) \uu_\e |^2\, dx dt\leq e^{(1+T)C( \Sigma)}.
\end{equation}
On the other hand, it follows from \eqref{normalize pdf} and   \eqref{projection1} that $\Pi_{\uu_\e } \p_t \uu_\e\parallel \p \dd_F (\uu_\e )$. So we can decompose  
\begin{align*}
\Big| \e \p_t \uu_\e   -\p \dd_F (\uu_\e )   (\nabla\cdot\bxi) \Big|^2=\Big| \e \p_t \uu_\e   -\e \Pi_{\uu_\e } \p_t \uu_\e  \Big|^2+\Big| \e \Pi_{\uu_\e } \p_t \uu_\e   -\p \dd_F (\uu_\e )   (\nabla\cdot\bxi) \Big|^2.
\end{align*}
This together with \eqref{energy1} yields
\begin{align}\label{energy2}
 \frac 1{\e ^2}\int_0^T\int_\O \Big| \e \p_t \uu_\e   -\e \Pi_{\uu_\e } \p_t \uu_\e  \Big|^2  \leq e^{(1+T)C( \Sigma)}.   \end{align}
The above   estimate  and   \eqref{energy bound0} imply  \eqref{energy bound4}. 
Concerning  \eqref{nablapsiest}, 
\begin{align*}
\int_\O |\nabla\psi_\e|\, dx& \overset{\eqref{projection}}\leq  \int_\O \(\frac \e 2|\Pi_{\uu_\e} \nabla \uu_\e|^2+\frac 1{2\e}|\p \dd_F(\uu_\e)|^2\)\, dx\\
& \overset{\eqref{gougudingli},\eqref{eq:2.7global}}\leq   A_\e(\uu_\e)\overset{\eqref{dissipation}}\leq A_\e (\uu_\e (\cdot,0)).
   \end{align*}
The estimate \eqref{calibration est1} follows from    \eqref{energy bound1}.
Finally,      \eqref{space der bound local} is a consequence of \eqref{energy bound3} and when it is combined  with \eqref{time est1}    leads us to \eqref{time der bound local}.
\end{proof}

We shall use \eqref{est away} together with the method of   Chen--Struwe \cite{MR990191} to show that the weak limits  of $\uu_\e$ are harmonic heat flows from the    bulk regions  $\O_t^\pm$ to $\mm_\pm$ respectively. However, the bulk potential  $F(\uu)$ (cf. \eqref{bulk potential}) depends on the relative distances to these two manifolds, and  we must find a quantitative way to distinguish them. This is done in the following:
\begin{theorem}\label{thm volume convergence}
Under the assumptions of Theorem \ref{main thm},  there exists $C_2>0$ independent of  $\e$ so that 
\begin{subequations}
\begin{align}
&\sup_{t\in [0,T]}B[\uu_\e  | \Sigma](t)\leq C_2\e,\label{volume convergenceweight}\\
&\sup_{t\in [0,T]}\int_\O| \psi_\e-c_F \1_{\O_t^+} |     \, dx\leq C_2\e^{1/2}.\label{volume convergence}
\end{align}
\end{subequations}
\end{theorem}

\begin{proof}
Within the proof, $h^\pm$ will denote   the positive/negative  part of a scalar function $h$. And we shall use the decomposition $h=h^+-h^-$. For $h\in W^{1,1}(\Omega)$, we have the following formula (cf. \cite[pp. 153]{MR3409135}): 
\[\p_{x_i} (h(x))^+= (\p_{x_i} h(x)) \1_{\{x:h(x)>0\}}(x)\quad \text{for } a.e.~~x.\label{der positive part}\]

The proof will be done in two steps. 

{\it Step 1: Derivation of differential inequalities.}
 Let $\chi(\cdot,t)=\1_{\O_t^+}-\1_{\O_t^-}$ and let $\eta(\cdot)$ be the  truncation  of the identity map
\[\eta(x):=\left\{
\begin{array}{rl}
x\qquad \text{when } &x \in [-\delta_0, \delta_0],\\
 \delta_0\qquad \text{when } &x\geq \delta_0,\\
 -\delta_0\qquad \text{when } &x\leq -\delta_0,
\end{array}
\right.\label{truncation eta}\]
 and   $\zeta:=|\eta|$.
 It follows from    \eqref{projection1}  and     the generalized  chain rule \eqref{ADM chain rule} that
%\[\p_{x_i} \psi_\e=\p_{x_i} \uu_\e\cdot\p \dd_F(\uu_\e)\] 
%As a result, we have the following identity
\begin{align}\label{volume evo1}
\p_t \psi_\e
=&\(\p_t \uu_\e+(\HH\cdot\nabla)\uu_\e\)\cdot \p \dd_F(\uu_\e) -\HH\cdot\nabla\psi_\e.
\end{align}

Motivated by the decomposition
 \[2\psi_\e-c_F=2(\psi_\e-c_F)^++ c_F -2(\psi_\e-c_F)^-,\label{psi deco}\]
 we shall establish differential inequalities  of the following  two  energies  which sum up to $B[\uu_\e  | \Sigma](t)$ (cf. \eqref{gronwall2new}):
\begin{subequations}\label{gronwall0}
\begin{align}
g_\e(t):=&\int  \( \psi_\e-c_F\)^+\zeta\circ\dd_\Sigma  \, dx,\label{gronwall1}\\
h_\e(t):=&\int   \Big(c_F\chi-c_F+ 2(\psi_\e-c_F)^- \Big)\eta\circ\dd_\Sigma \, dx.\label{gronwall2}
\end{align}
\end{subequations}
Since $\psi_\e\geq 0$, we have $(\psi_\e-c_F)^-\in [0,c_F]$ and thus $c_F -2(\psi_\e-c_F)^-$ ranges in $ [-c_F,c_F]$. 
 In view of \eqref{truncation eta}, we have $ c_F \chi \eta\circ \dd_\Sigma\geq 0$, so  the integrands of these two  energies are both non-negative.
Moreover,  \eqref{initial} implies that  
\[g_\e(0)+h_\e(0)\lesssim \e.\] 
Now we proceed in the derivation of G\"{o}nwall's inequalities of $g_\e$ and $h_\e$.
Using \eqref{volume evo1}
\begin{align*}
g_\e'(t)
\overset{\eqref{volume evo1}}=&\int_{\{ \psi_\e > c_F\}} (\p_t \uu_\e+(\HH\cdot\nabla)\uu_\e)\cdot \p \dd_F(\uu_\e) \zeta(\dd_\Sigma) \\
&-\int_{\{ \psi_\e > c_F\}} \HH\cdot \nabla\psi_\e \zeta(\dd_\Sigma) +\int  \(\psi_\e-c_F\)^+ \p_t \zeta(\dd_\Sigma) \\
\overset{\eqref{der positive part} }{=}&\int_{\{ \psi_\e > c_F\}} (\p_t \uu_\e+(\HH\cdot\nabla)\uu_\e)\cdot\p \dd_F(\uu_\e)\zeta(\dd_\Sigma) \\
&-\int  \HH\cdot \nabla \(\psi_\e-c_F\)^+ \zeta(\dd_\Sigma) -\int   \(\psi_\e-c_F\)^+ \HH\cdot\nabla \zeta(\dd_\Sigma) \\
&+\int  \(\p_t \zeta(\dd_\Sigma)+\HH\cdot\nabla \zeta(\dd_\Sigma)\)  \(\psi_\e-c_F\)^+\end{align*}
Finally by an integration by part, we can merge   the second and the third integral in the last display:
 \begin{align*}
g_\e'(t)
\overset{\eqref{eq:2.7global}}\leq &\int_{\{ \psi_\e > c_F\}} \left|(\p_t \uu_\e+(\HH\cdot\nabla)\uu_\e)\cdot\frac{\p \dd_F(\uu_\e)}{|\p \dd_F(\uu_\e)|}  \sqrt{2F(\uu_\e)}\right| \zeta(\dd_\Sigma) \\
&+\int  (\div \HH) \(\psi_\e-c_F\)^+ \zeta(\dd_\Sigma)  +\int  \(\p_t \zeta(\dd_\Sigma)+\HH\cdot\nabla \zeta(\dd_\Sigma)\)  \(\psi_\e-c_F\)^+\nonumber\\
&\overset{\eqref{mcf}}\leq \int \e \Big|\p_t \uu_\e+(\HH\cdot\nabla)\uu_\e\Big|^2+ \int   \frac 1{\e}{F(\uu_\e)}\zeta^2(\dd_\Sigma)  +Cg_\e(t)\\
&\overset{ \eqref{energy bound3}}\leq \int \e \Big|\p_t \uu_\e+(\HH\cdot\nabla)\uu_\e\Big|^2+CE_\e[\uu_\e |\Sigma] +Cg_\e(t). \end{align*}
 In view of \eqref{time est1}, we can apply  Gr\"{o}nwall's lemma and obtain $g_\e(t)\leq C \e$.

Similar calculation shows 
$h_\e'(t)\leq C h_\e(t)$. For simplicity we denote $$c_F\chi-c_F+ 2(\psi_\e-c_F)^-=:w_\e.$$ Using $\p_i \chi\eta (\dd_\Sigma)\equiv 0$ (in the sense of  distributions), we find 
  \[\p_i w_\e \eta(\dd_\Sigma)=2\p_i \psi_\e \1_{\{\psi_\e< c_F\}}\eta(\dd_\Sigma)\quad \text{for } a.e.~~x.\] So by the same  calculation for $g_\e$  we obtain
\begin{align*}
h_\e'(t)
\leq &\int_{\{ \psi_\e < c_F\}} 2\left|(\p_t \uu_\e+(\HH\cdot\nabla)\uu_\e)\cdot\frac{\p \dd_F(\uu_\e)}{|\p \dd_F(\uu_\e)|}  \sqrt{2F(\uu_\e)}\right| \zeta(\dd_\Sigma)  \\
&+\int  (\div \HH)w_\e\eta(\dd_\Sigma)  +\int  \Big(\p_t \eta(\dd_\Sigma)+\HH\cdot\nabla \eta(\dd_\Sigma)\Big)  w_\e\\
\leq & \int \e \Big|\p_t \uu_\e+(\HH\cdot\nabla)\uu_\e\Big|^2+C E_\e[\uu_\e |\Sigma]  +Ch_\e(t),\end{align*}
and $h_\e(t)\leq C\e$ follows from  the Gr\"{o}nwall lemma, and we thus proves \eqref{volume convergenceweight}.
Finally,
\begin{align}
&\int |2\psi_\e-c_F-c_F\chi|\zeta(\dd_\Sigma)\nonumber\\
& \overset{\eqref{psi deco}}\leq   \int 2(\psi_\e-c_F)^+\zeta(\dd_\Sigma) + \int \Big|c_F -2(\psi_\e-c_F)^--c_F\chi\Big| \zeta(\dd_\Sigma)\nonumber\\
 & =  2g_\e+h_\e\leq C\e.\label{gronwall3}
\end{align}

{\it Step 2: Pass to the  unweighted inequality.}
We first note that \eqref{gronwall3} implies   $\eqref{volume convergence}$ with $\O$    replaced by  $\O\backslash B_{\delta_0}(\Sigma_t)$. So we shall focus on the estimate in $B_{\delta_0}(\Sigma_t)$. 
We shall  use the following elementary estimate  
\[
\left(\int_{0}^{\delta_0}|f(y)| \, dy\right)^2  \leq 2\|f\|_{\infty} \int_{0}^{\delta_0}|f(y)| y \, dy\quad  \forall f \in L^{\infty}\left(0, \delta_0\right).
\]
Let $\chi_\e=\frac{2\psi_\e-c_F}{c_F}$. It follows from \eqref{L infinity bound1} that $\| \chi_\e\|_{\infty}$ is uniformly bounded.
For each fixed $p\in \Sigma_t$ and $y\in [-\delta_0,\delta_0]$, we have $|\dd_\Sigma| (p+y \nn(p), t  )=|y|$. So we can apply the above inequality to estimate 
\begin{align*}
&\(\int_{0}^{\delta_0}\left|\chi (p+ y \nn, t )-\chi_\e (p+ y \nn, t )\right| \, dy\)^2\\
\leq & 2\(1+\| \chi_\e\|_{\infty}\)\int_{0}^{\delta_0}\left|\chi (p+y \nn, t )-\chi_\e (p+y \nn, t )\right| |\dd_\Sigma| (p+y \nn, t  ) \, dy.
\end{align*}
So by area formula, 
\begin{align*}
& \(\int_{B_{\delta_0}(\Sigma_t) }\left|\chi(x,t)-\chi_\e(x, t)\right| \, dx\)^2  \\
=&  ~  \(\sum_{\pm} \int_{\Sigma_t } \int_{0}^{\delta_0}\left| \chi\left(p\pm y \nn, t\right)-\chi_\e (p\pm y \nn, t )\right| \, dy    d S(p)\)^2  \\
\leq  & ~C\int_{\Sigma_t }\int_{-\delta_0}^{\delta_0}\left|\chi\left(p+y \nn, t\right)-\chi_\e\left(p+y \nn, t\right)\right| |\dd_{\Sigma}|\left(p+y \nn,  t \right) \, dy d S(p) \\
=& ~ \int_{B_{\delta_0}(\Sigma_t) }\left|\chi(x, t)-\chi_\e(x, t)\right| |\dd_\Sigma(x,  t )| \, dx.
\end{align*}
This implies  the estimate  in $B_{\delta_0}(\Sigma_t)$.  
\end{proof}

\begin{corollary}\label{global control prop}
There exists a sequence of $\e_k\downarrow 0$  and    $ \uu^\pm(x,t)$  
so that 
\begin{align}\label{u equal 0 region}
 \uu^\pm\in L^\infty(0,T;     L^\infty(\O)\cap H^1_{loc}(\O^+_t;\mm_\pm)) 
,\p_t \uu^\pm\in  L^2(0,T;  L^2_{loc}(\O^+_t)),
\end{align}
and 
\begin{subequations}\label{weak strong convergence}
\begin{align}
\p_t \uu_{\e_k}\xrightarrow{ k\to\infty } \p_t \uu^\pm   &~\text{weakly in}~  L^2(0,T;L^2_{loc}(\O^\pm_t)),\label{deri con2}\\
\nabla \uu_{\e_k}\xrightarrow{k\to\infty }  \nabla \uu^\pm   &~\text{weakly in}~  L^\infty(0,T;L^2_{loc}(\O^\pm_t)),\label{deri con}\\
  \uu_{\e_k}\xrightarrow{k\to\infty }    \uu^\pm   & ~\text{strongly in}~  C([0,T];L^2_{loc}(\O^\pm_t ;\mm_\pm)).\label{deri con1}
\end{align}
\end{subequations}
\end{corollary}
\begin{proof}
It follows from \eqref{L infinity bound1} and  \eqref{est away}   that,  for any $\delta\in (0,\delta_0)$,   there exists a subsequence $\e_k=\e_k(\delta)>0$ such that
\begin{subequations}\label{udelta conv}
\begin{align}
  \uu_{\e_k}\xrightarrow{k\to\infty }  \uu^\pm&~\text{weakly-star in}~  L^\infty(0,T ;L^\infty(\O))\label{convergence L4},\\
\p_t \uu_{\e_k}\xrightarrow{ k\to\infty } \p_t \bar{\uu}_\delta^\pm&~\text{weakly in}~  L^2(0,T;L^2(\O^\pm_t\backslash B_\delta(\Sigma_t))),\label{convergence weak time der}\\
\nabla \uu_{\e_k}\xrightarrow{k\to\infty } \nabla \bar{\uu}_\delta^\pm&~\text{weakly-star in}~  L^\infty(0,T;L^2(\O^\pm_t\backslash B_\delta(\Sigma_t))),\label{convergence weak gradient}
  \end{align}
\end{subequations}
 and      $\uu^\pm=\bar{\uu}_\delta^\pm$ a.e. in $U_\pm(\delta):=\cup_{t\in [0,T]} \(\O_t^\pm\backslash B_\delta(\Sigma_t)\)\times \{t\}$. By the arbitrariness of $\delta$ we deduce 
\begin{equation}\label{regular limit}
\uu^\pm\in L^\infty(0,T;H^1_{loc}(\O^\pm_t))   ~\text{with}~\p_t \uu\in L^2(0,T;L^2_{loc}(\O^\pm_t)).
\end{equation}
Moreover, by a diagonal argument   we obtain \eqref{deri con2} and \eqref{deri con}, and also    \eqref{deri con1} by the  Aubin--Lions lemma.

It remains to show that $\uu^\pm$ are mappings into $\mm_\pm$. 
 Using \eqref{deri con1}, \eqref{space der bound local},  and Fatou's lemma, we deduce that $F(\uu^\pm)=0$ a.e. in $\O$. In view of  \eqref{limit manifold} we deduce    that  the images of $\uu^\pm$ lie  in $\mm=\mm_+\sqcup  \mm_-$. Owning to  \eqref{deri con1} and \eqref{L infinity bound1},
 \[\psi_{\e_k}\overset{\eqref{psi}}=\dd_F\circ \uu_{\e_k} \xrightarrow{k\to\infty} \dd_F\circ \uu^\pm  \text{ strongly in }C([0,T];L^2_{loc}(\O^\pm_t)).\]
 This together  with  \eqref{volume convergence} and  \eqref{eq:1.6} yields that $\uu^\pm$ maps into $\mm_\pm$ respectively.  Combining this with    \eqref{regular limit} yields \eqref{u equal 0 region}.
   \end{proof}

\begin{lemma}\label{area control}
For any $\delta\in (0,\delta_0)$, there exist   $b^\pm_\delta \in [ \delta,2\delta]$ s.t. the sets 
\[\{x:   \psi_\e  >  c_F-b^+_\delta \}\text{ and }\{x:\psi_\e  <b^-_\delta\}\]  have    finite perimeters and 
\begin{subequations}
\begin{align}
   \left|\mathcal{H}^{d-1}\(\{x:\psi_\e  =c_F-b^+_\delta\}\)-\mathcal{H}^{d-1} (\Sigma_t)\right| & \leq   C \e^{1/2}\delta^{-1},\label{area compare8}
 \\  \left|\mathcal{H}^{d-1}\(\{x:\psi_\e  =b^-_\delta\}\)-\mathcal{H}^{d-1} (\Sigma_t)\right| & \leq  C \e^{1/2}\delta^{-1}.\label{area compare}
\end{align}
\end{subequations}
\end{lemma}
\begin{proof}
For any $\delta<\delta_0\ll c_F$,  we denote (within the proof of the lemma) 
\[\O_t^{\e,\delta}=\{x\in \O:c_F-2\delta< \psi_\e(x,t) <  c_F-\delta\}.\]   
We shall also denote the $d$-Lebesgue's measure of a set $A$ by $|A|$.

 Using \eqref{calibration est1} and the co-area formula (cf. \cite[section 5.5]{MR3409135}), we deduce that for almost every $\delta \in (0,\delta_0)$,
\begin{align*}
&C\e   \overset{\eqref{calibration est1}}\geq   \int_{ \O_t^{\e,\delta}} \(|\nabla\psi_\e|- \bxi\cdot \nabla\psi_\e\)\, dx \qquad (\geq 0)\\
&=\int_{c_F -2\delta}^{c_F-\delta}   \mathcal{H}^{d-1}\(\{x:\psi_\e  =s\}\)\, ds-  \int_{ \p \O_t^{\e,\delta}} \bxi\cdot \nu \psi_\e   \, d\mathcal{H}^{d-1}  +\int_{ \O_t^{\e,\delta}} (\div \bxi) \psi_\e  \, dx.
\end{align*}  
where $\nu$ is the outward unit normal of the set under integration, defined on its (measure-theoretic) boundaries. Note that $\|\psi_\e\|_{\infty}$ is uniformly bounded due to \eqref{L infinity bound1}.
So we can estimate 
\begin{align}
 &\Big|\int_{c_F -2\delta}^{c_F-\delta}    \mathcal{H}^{d-1}\(\{x:\psi_\e  =s\}\)\, ds -  \int_{ \p \O_t^{\e,\delta}} \bxi\cdot \nu \psi_\e   \, d\mathcal{H}^{d-1}\Big|\nonumber\\ 
 \leq ~& C  (\e  + \|\psi_\e\|_{\infty}  |\O_t^{\e,\delta}|)\leq C\e^{1/2}.   \label{volume est1}
\end{align}
where we use the Chebyshev  inequality and \eqref{volume convergence}  in the last step.
On the other hand, applying  the divergence theorem and  adding zero,
\begin{align*}
 \int_{ \p \O_t^{\e,\delta}} \bxi\cdot \nu \psi_\e   \, d\mathcal{H}^{d-1} 
=&  \(c_F-2\delta \) \int_{ \{\psi_\e> c_F-2\delta\}} \div \bxi    \, dx+  (c_F-\delta) \int_{ \{\psi_\e< c_F- \delta\}}  \div \bxi    \, dx,\\
-\delta  \mathcal{H}^{d-1} (\Sigma_t)=&-\(c_F-2\delta \)\int_{\O_t^+}\div \bxi\, dx-(c_F-\delta) \int_{\O_t^-}\div \bxi\, dx.
\end{align*}
 Adding the above two equations and substituting  into \eqref{volume est1}, we obtain  
\begin{align}\label{volume est3}
&\left|  \int_{c_F -2\delta}^{c_F-\delta}   \mathcal{H}^{d-1}\(\{x:\psi_\e  =s\}\)\, ds-\delta  \mathcal{H}^{d-1} (\Sigma_t)\right|\nonumber\\
&\leq C\(  \e^{1/2} +  \Big|  \O_t^+\triangle     {\{x: \psi_\e> c_F-2\delta\}} \Big|  +  \Big|  \O_t^- \triangle    {\{x: \psi_\e < c_F -\delta\}} \Big|\),
\end{align}
where $A\triangle B=(A-B)\cup (B-A)$ is the symmetric difference of two sets $A,B$. 
We rewrite the last two terms by 
\begin{align*}
r_\e^+:=  &~\Big|  \O_t^+\triangle     {\{x: \psi_\e> c_F-2\delta\}} \Big|\\
=&~\Big|     {\{x\in \O_t^+: \psi_\e\leq  c_F-2\delta\}} \Big|+  \Big|   {\{x\in \O_t^-: \psi_\e >c_F-2\delta\}} ,\\
r_\e^-:=  &~\Big|  \O_t^-\triangle     {\{x: \psi_\e < c_F- \delta\}} \Big|\\
=&~\Big|     {\{x\in \O_t^-: \psi_\e\geq   c_F- \delta\}} \Big|+  \Big|   {\{x\in \O_t^+: \psi_\e < c_F-\delta\}} \Big|.\end{align*} 
Now using  the Chebyshev  inequality and \eqref{volume convergence}  we get     
$r_\e^-+r_\e^+\leq C \e^{1/2}$. Substituting this estimate into \eqref{volume est3} leads  to 
\begin{align}\label{volume est4}
&\left| \frac 1 \delta   \int_{c_F-2 \delta}^{c_F-\delta}   \mathcal{H}^{d-1}\(\{x:\psi_\e  =s\}\)\, ds- \mathcal{H}^{d-1} (\Sigma_t)\right|\leq C\e^{1/2}\delta^{-1}.
\end{align}
So the existence of $b^+_\delta\in [\delta,2\delta]$ satisfying  \eqref{area compare8}  follows from  Fubini's theorem. The inequality \eqref{area compare} can be done in the same way and we omit the proof.\end{proof}

 \begin{prop}\label{prop estimates}
Let   $\e_k\downarrow 0$ be the sequence in Corollary \ref{global control prop}. Then there exists   $b^\pm_k \in [\e_k^{1/6},2\e_k^{1/6}]$ so that the sets
\begin{subequations}\label{omegasets}
\begin{align}
\O_t^{k,+}&:=\{x\in \O:      \psi_{\e_k}(x,t) > c_F-b^+_k\},\\
\O_t^{k,-}&:=\{x\in \O: \psi_{\e_k}(x,t)< b^-_k\}
\end{align}
\end{subequations}
have uniformly bounded  perimeter. Moreover,  up to the extraction of a subsequence, we have 
\begin{subequations}
\begin{align}
  &\left|\mathcal{H}^{d-1}(\p  \O_t^{k,\pm})-\mathcal{H}^{d-1} (\Sigma_t)\right| \leq  C \e_k^{1/3},\label{area compare1}\\
 & \1_{\O_t^{k,\pm}}\xrightarrow{k\to \infty} \1_{\O_t^\pm}\text{ weakly-star in }BV(\O),\label{area converge}\\
  & \p \O_t^{k,\pm} \xrightarrow{k\to \infty}	 \Sigma_t\quad \text{ under Hausdorff metric }.\label{hausdorff conv}
\end{align}
\end{subequations}
Finally there exists   $K_1\in\mathbb{N}^+$ so that for any $k>K_1$, the solution   $\uu_{\e_k}$ satisfies
\begin{subequations}
\begin{align}
 &\uu_{\e_k}(\O_t^{k,\pm})\subset   B_{\delta_0}(\mm_\pm)\label{ukin ball},\\
\sup_{t\in  [0,T]}  & \int_{\O_t^{k,\pm}} \Big|  \nabla P_{\mm} (\uu_{\e_k})   \Big|^2\, dx+ \int_0^T\int_{\O_t^{k,\pm}}  \Big|  \p_t  P_{\mm} (\uu_{\e_k})   \Big|^2\, dxdt\leq C.\label{local energy}
 \end{align}
\end{subequations}
\end{prop}
\begin{proof}
Choosing $\delta=\delta_k:=\e_k^{1/6}$  in Lemma \ref{area control}  yields $b_k^+ \in  [\e_k^{1/6},2\e_k^{1/6}]$ so that
\[ \left|\mathcal{H}^{d-1}\(\{x:\psi_\e  =c_F-b^+_k\}\)-\mathcal{H}^{d-1} (\Sigma_t)\right|\leq \e_k^{1/3}.\]
This  leads to    the `plus' case of \eqref{area compare1} and the `minus' case can be done in the same way. By   \eqref{volume convergence},
\[\1_{\{x:\psi_\e   > c_F-b^+_k\}}\xrightarrow{\e\to 0}\1_{\O_t^+} \text{ strongly in } L^1(\O),~\text{ for each fixed }k.\]
By a diagonal argument,
 we find a subsequence  of $\e_k$ (without relabeling)   so    that 
\[\1_{\O_t^{k,+}}\xrightarrow{k\to \infty} \1_{\O_t^+}\text{ strongly in } L^1(\O).\]
This combined with \eqref{area compare1} implies the `plus' case of   \eqref{area converge}.
The `minus' cases can be done in a  same way. The convergence  \eqref{hausdorff conv} is a consequence of the Blaschke's Theorem (cf. \cite[Chapter 7]{MR1835418}).
 By \eqref{quasidistance}, \eqref{eq:1.6} and $b_k^\pm\to 0$, there exists $K_1>0$ so that for any $k\geq K_1$ there holds
\begin{align*}
\dd_F\circ \uu_{\e_k}>c_F-b_k^+\text{ implies } \uu_{\e_k}\in B_{\delta_0}(\mm_+),\\
\dd_F\circ \uu_{\e_k}< b_k^-\text{ implies } \uu_{\e_k}\in B_{\delta_0}(\mm_-),
\end{align*}
and thus  \eqref{ukin ball} holds. 

It remains to use  \eqref{ukin ball} to derive \eqref{local energy}. So we shall always assume $k\geq K_1$.
If we denote the co-dimension of  $\mm$ to be $\ell\in \mathbb{N}^+$, then as  
  the nearest point projection $P_{\mm}$  is smooth in $B_{\delta_0}(\mm)$,
any vector  $\uu\in B_{\delta_0}(\mm)$ can be written as 
\[\uu=P_{\mm} \uu+\dd_{\mm}(\uu) \bnu_{\mm}(  \uu)=P_{\mm} \uu+\sum_{j=1}^\ell \dd_j(\uu) \bnu_j(P_{\mm}\uu)\label{local projection}\] 
where $\{\bnu_j\}_{j=1}^\ell$ is  an orthonormal frame of the normal space at $P_{\mm}\uu$ and $\{\dd_j(\uu)\}_{j=1}^\ell$ is the   coordinate of $\uu_{\e_k}-P_{\mm} \uu_{\e_k}$ in such a    frame.   
So we have 
\[\dd^2_{\mm}(\uu)= \sum_{j=1}^\ell \dd_j^2(\uu)~\text{ and }~ \p\dd_j(\uu)=\bnu_j(P_{\mm}\uu),\quad  \forall \uu\in B_{\delta_0}(\mm).\label{multi distance}\] Note that $\{\dd_j\}_{j=1}^\ell\subset  C^1(B_{\delta_0}(\mm))$ and    $\dd_{\mm}(\uu)\in C^1(B_{\delta_0}(\mm)\backslash \mm)$, and in general $\dd_\mm$ is not differentiable on $\mm$. 
  On the (open) set  $\{x\mid \uu_{\e_k}(x,t)\in B_{\delta_0}(\mm)\}$, we can 
differentiate \eqref{local projection} and get
\[\p_{x_i} \uu_{\e_k}=\p_{x_i}  \(P_{\mm} \uu_{\e_k}\)+\sum_{j=1}^\ell \dd_j(\uu_{\e_k})\p_{x_i} \bnu_j(P_{\mm}\uu_{\e_k})+\sum_{j=1}^\ell \p_{x_i} \dd_j(\uu_{\e_k}) \bnu_j(P_{\mm}\uu_{\e_k})\label{decom near m}\]
   for $0\leq i\leq d$.
Note that the first two   terms are tangential to $\mm$ while the last one is  orthogonal to $\mm$. So  we have 
\begin{align*}
|\p_{x_i} \uu_{\e_k}|^2& \geq |\p_{x_i}  (P_{\mm} \uu_{\e_k})|^2+\left|\sum_{j=1}^\ell \p_{x_i} \dd_j(\uu_{\e_k}) \bnu_j(P_{\mm}\uu_{\e_k})\right|^2\nonumber\\
&\quad +2 \p_{x_i}  \(P_{\mm} \uu_{\e_k}\)\cdot \sum_{j=1}^\ell \dd_j(\uu_{\e_k})\p_{x_i} \bnu_j(P_{\mm}\uu_{\e_k})\nonumber\\
&=  |\p_{x_i}  (P_{\mm} \uu_{\e_k})|^2 +\sum_{j=1}^\ell \left|\p_{x_i} \dd_j(\uu_{\e_k}) \right|^2\nonumber\\
&\quad +2\sum_{j=1}^\ell \dd_j(\uu_{\e_k}) A(P_{\mm}\uu_{\e_k})\Big( \p_{x_i}  \(P_{\mm} \uu_{\e_k}\),\p_{x_i}  \(P_{\mm} \uu_{\e_k}\)\Big)
\end{align*}
where $A(P_{\mm}\uu_{\e_k})(\cdot,\cdot)$ is the second fundamental form of $\mm$ at $P_{\mm}\uu_{\e_k}$. This combined with the second formula of \eqref{multi distance} yields  
\begin{align}
|\p_{x_i} \uu_{\e_k}|^2& \geq (1-C\delta_0) |\p_{x_i}  (P_{\mm} \uu_{\e_k})|^2+\sum_{j=1}^\ell \left|\p_{x_i}  \uu_{\e_k}\cdot  \bnu_j(P_{\mm}\uu_{\e_k})\right|^2, \label{nonlinear decom}
\end{align}
where $C>0$ is a constant depending on the geometry of  $\mm$.
On the other hand, by  the first case of   \eqref{projection1}  and \eqref{local projection},
\begin{align*}
|\Pi_{\uu_{\e_k}} \p_{x_i}  \uu_{\e_k}|^2 \dd_\mm^2(\uu_{\e_k})&=|\p_{x_i} \uu_{\e_k}\cdot \bnu_{\mm}(\uu_{\e_k})\dd_\mm(\uu_{\e_k}) |^2\nonumber\\
&=\left|\sum_{j=1}^\ell \dd_j(\uu_{\e_k})\p_{x_i} \uu_{\e_k}\cdot  \bnu_j(P_{\mm}\uu_{\e_k})\right |^2\nonumber\\
&\leq\sum_{j=1}^\ell \dd_j^2(\uu_{\e_k}) \sum_{j=1}^\ell \left|\p_{x_i} \uu_{\e_k}\cdot  \bnu_j(P_{\mm}\uu_{\e_k})\right |^2.
\end{align*}
Using the first equation of \eqref{multi distance}, we obtain $|\Pi_{\uu_{\e_k}} \p_{x_i}  \uu_{\e_k}|^2\leq \sum_{j=1}^\ell \left|\p_{x_i} \uu_{\e_k}\cdot  \bnu_j(P_{\mm}\uu_{\e_k})\right |^2$. 
Subtracting this inequality    from \eqref{nonlinear decom}, and using orthogonality of the projection \eqref{projection1}, we obtain 
\begin{align}\label{linear control nonlinar}
 (1-C\delta_0)|\p_{x_i}  (P_{\mm} \uu_{\e_k})|^2 \leq    |\p_{x_i} \uu_{\e_k} -\Pi_{\uu_{\e_k}} \p_{x_i}  \uu_{\e_k}|^2,\qquad 0\leq i\leq d.      \end{align}
   Choosing $\delta_0$ sufficiently small in \eqref{linear control nonlinar} and using  \eqref{energy bound4}, we obtain  \eqref{local energy}.
 
\end{proof}

\begin{theorem}\label{thm usesbv}
The   mappings 
\[\uu^\pm:\bigcup_{t\in [0,T]} \O_t^\pm \times \{t\}\mapsto \mm_\pm\]  obtained in Corollary \ref{global control prop}   are weak solutions of  harmonic heat flows that satisfy 
\begin{align}
\uu^\pm\in L^2(0,T; H^1(\O_t^\pm;\mm_\pm))\label{upm regu}
\end{align}
additionally.
Moreover,  with the notations in  Proposition \ref{prop estimates},  the functions 
\[\vv_k(\cdot,t):=\sum_\pm P_{\mm_\pm} \circ  \uu_{\e_k}(\cdot,t) ~\1_{\O_t^{k,\pm} }\]satisfy the following properties for a.e. $t\in [0,T]$:
 \begin{subequations}
\begin{align}
 \vv_k(\cdot,t) &\xrightarrow{k\to\infty} \uu=\sum_{\pm}\uu^\pm(\cdot,t) ~\1_{\O_t^\pm } \text{ weakly-star in  } BV(\O)\label{strong convergence of u1},\\
\nabla^{a} \vv_k   & \xrightarrow{k\to\infty}  \1_{\O_t^\pm} \nabla \uu^\pm \text{   weakly in } L^1(\O),\\
\sum_\pm \int_{\O_t^\pm} & |\nabla \uu^\pm(\cdot,t)|^2\, dx\leq  \liminf_{k\to \infty}\sum_\pm \int_{\O_t^{k,\pm}} \left|\nabla^a \vv_k(\cdot,t)\right|^2\,dx.\label{lower semi sbv}
\end{align}
\end{subequations}
Here in \eqref{lower semi sbv}  $\nabla^{a} \vv_k$ is the absolute continuous part of the distributional gradient $\nabla$.
\end{theorem}
\begin{proof}
The sequence $\vv_k(\cdot,t)$  is  bounded in  $L^\infty(\O)$, and by \eqref{local energy} we deduce that  their   distributional derivatives  have no Cantor parts. Moreover, the   absolute continuous parts and the jump sets enjoy the estimates \eqref{local energy} and \eqref{area compare1} respectively.
So it follows from  Proposition \ref{AFP2}  that   $\{\vv_k(\cdot,t)\}$ is compact in $SBV(\O)$: there exists $\vv\in SBV(\O)$ so that  $\vv_k\to\vv$ weakly-star in  $BV(\O)$ as $k\to \infty$,
and the absolute continuous  part of the gradient 
\[\nabla^{a} \vv_k= \sum_\pm \nabla P_{\mm_\pm} (\uu_{\e_k}) ~\1_{\O_t^{k,\pm} }  \xrightarrow{k\to\infty}  \nabla^{a} \vv \text{   weakly in } L^1(\O).\]  
To identify $\vv$, we   combine \eqref{area converge} with \eqref{deri con1} and deduce that   $\vv=\sum_\pm \1_{\O_t^\pm}\uu^\pm$ a.e. in $\Omega$, and thus   \eqref{strong convergence of u1} is proved.   The lower semicontinuity of $SBV$ functions (cf. \eqref{LSC sbv}) implies  \eqref{lower semi sbv}, and thus we can improve the spatial regularity in  \eqref{u equal 0 region} to \eqref{upm regu}.
Finally combining \eqref{est away}, \eqref{weak strong convergence} with    Chen--Struwe \cite{MR990191} imply that, $\uu^\pm$  are weak solutions to   harmonic heat flows respectively.
\end{proof}
 
%%%%%%%%%%%%%%%%%%%%%%%%%%%%%%%%%%%%%%%%%%%%%%%%%%%%%%%%%%%%%%%%%%%%%%%%%%%%%%%%%%%%%%%%%%%%%%%%%%%%%%%%%%%%
  \section{Proof of Theorem \ref{main thm}}\label{sec mp}  
  
     We first recall  that the estimate  \eqref{intro cali} is proved  in Lemma \ref{lemma level} (cf. \eqref{calibration est2}). The estimates \eqref{intro modulodist} and  \eqref{volume convergencethm} are obtained in Theorem
    \ref{thm volume convergence}.
 %Now we prove    \eqref{intro energy conv}: using \eqref{bc n and H} and \eqref{def:xi}, for every $t\in [0,T]$, 
    %\begin{align} 
 %    &  \left| \int_\O  \bxi\cdot\nabla \psi_\e\, dx-c_F\mathcal{H}^{d-1}(\Sigma_t)\right|\nonumber\\
 %=&\left|-\int_\O (\div  \bxi) \,  \psi_\e\, dx+\int_{\O} (\div \bxi) c_F\1_{\O_t^+}\, dx\right| \overset{\eqref{volume convergencethm}}\leq C(\bxi) \e^{1/2}.
%\end{align}
%This combined with 
%\eqref{calibration est2}  yields \eqref{intro energy conv}.
    The convergence \eqref{strong global of Q} is obtained  in Corollary \ref{global control prop}. The limit $\uu^\pm$ being harmonic heat flow with regularity \eqref{reg limit} has been done in Theorem \ref{thm usesbv}.
It remains  to prove the minimal pair boundary conditions \eqref{thm minimal pair}.  To this end, we introduce the semi-distance function
\begin{align}\label{semidistance1}
 \dd_F^*(\vv_+,\vv_-)&=\inf_{\substack{\bxi(a_\pm)=\vv_\pm\\ \bxi\in H^1((a_-,a_+),\R^n)}} \int_{a_-}^{a_+} |\bxi'(t)|\sqrt{2F(\bxi(t))}\, dt.
\end{align}
for any $-\infty\leq a_-<a_+\leq \infty$.
Note that the above definition is independent of the choice of $a_\pm$ (cf. \cite{Lin2012a}).
Such a function can be used to define a semi-distance between closed sets $S_\pm\subset  \R^n$:
\[ \dd_F^*(S_+, S_-)=\inf_{\vv_\pm\in S_\pm}  \dd_F^*(\vv_+,\vv_-).\]
Let $P_{\mm_\pm}$ be the nearest point projection to $\mm_\pm$. For any subset $S_\pm\subset \mm_\pm$, we define
\begin{align}
\mathcal{N}(S_\pm,\rho):=\bigcup_{\uu_\pm\in S_\pm}\mathcal{N}(\uu_\pm,\rho),
\end{align}
where $\mathcal{N}(\uu_\pm,\rho)$ is the `normal sphere' of radius $\rho$ centered at $\uu_\pm\in \mm_\pm$ under the metric $ \dd_F^*$:
\[\mathcal{N}(\uu_\pm,\rho):=\{\uu\in\R^n:  \dd_F^*(\uu, \uu_\pm)= \rho,P_{\mm_\pm} \uu=\uu_\pm\}.\label{def normalsphere}\]
%Here $\mathcal{N}_{\uu_\pm}\mm_\pm$ is the normal space of $\mm_\pm$ at $\uu_\pm$ respectively.
 Inspired by    \cite{wangnote}, we  introduce a   function $\kappa:\mm_+\times \mm_-\times [0,1]\to \R$ by 
 \[\kappa(\uu_+,\uu_-,\rho):= \dd_F^* \Big(\mathcal{N}\(\bar{B}_{\rho}(\uu_+)\cap \mm_+,\delta_0\),\mathcal{N}\(\bar{B}_{\rho}(\uu_-)\cap \mm_-,\delta_0\)\Big)+ 2\delta_0-c_F,\label{lowerestI0}\]
 where $c_F=2 \int_{0}^{\frac {\dist_{\mm}}2} \sqrt{2f(\lambda^2)} d \lambda$ (cf. \eqref{linpanwang cf equ}).
It is obvious that $\kappa$ is continuous w.r.t all its variables and non-increasing w.r.t $\rho$.  Note that $\mathcal{N}\(\bar{B}_{\rho}(\uu_+)\cap \mm_+,\delta_0\)$ can be visualized as a tube of thickness $\delta_0$ centered at $\bar{B}_{\rho}(\uu_+)\cap \mm_+$.
\begin{lemma}\label{wanglemma1}
$\kappa(\uu_+,\uu_-,\rho)$ is non-negative. Moreover, $\kappa(\uu_+,\uu_-,0)=0$ if and only if 
  $(\uu_+,\uu_-)\in \mm_+\times \mm_-$ is   a minimal pair.  

\end{lemma}
\begin{proof}
It is easy to check that $ \dd_F^*(\vv_+,\vv_-)$ is continuous, and it vanishes  if and only if either $\vv_+=\vv_-$ or $\vv_\pm$ both lie in one of $\mm_\pm$.  
So if we identify the sets $\mm_\pm$ as two  points, the resulting quotient space $(\R^n/\mm_\pm, \dd_F^*)$ is a metric space.

For any $\vv_\pm \in \mathcal{N}\(\bar{B}_{\rho}(\uu_\pm)\cap \mm_\pm,\delta_0\)$, by triangle inequality 
\begin{align*}
 \dd_F^*(\vv_+,\vv_-)+2\delta_0=  \dd_F^*(\vv_+,\vv_-)+\sum_{\pm} \dd_F^*(\vv_\pm,P_{\mm}\vv_\pm)\geq  \dd_F^*( P_{\mm}\vv_+,P_{\mm}\vv_-)=c_F.
\end{align*}
Minimizing $\vv_\pm$   implies that $\kappa(\uu_+,\uu_-,\rho)$ is non-negative.

If   $(\uu_+,\uu_-) $ is   a minimal pair, then   the line segment $\overline{\uu_-\uu_+}$ meets $\mm_\pm$ perpendicularly.  Let  $\vv_\pm=\overline{\uu_-\uu_+}\cap \mathcal{N}\( \uu_\pm ,\delta_0\)$. Then
\begin{align*}
\kappa(\uu_+,\uu_-,0) =   \dd_F^* (\vv_+,\vv_-)+ \sum_{\pm} \dd_F^*(\vv_\pm,\uu_\pm)-c_F=0.
\end{align*}

Now we assume   $(\uu_+,\uu_-) $ is NOT  a minimal pair, i.e. $|\uu_+-\uu_-|> \dist_\mm$.    We   claim that  
\[\mathcal{N}(\uu_+,\tfrac 12c_F)\cap \mathcal{N}(\uu_-,\tfrac 12 c_F)=\varnothing.\label{ringnonoverlap}\]  
 If \eqref{ringnonoverlap}   were wrong, there would exist $\uu\in \cap_\pm \mathcal{N}(\uu_\pm,\tfrac 12c_F)$. 
 It follows from  \eqref{def normalsphere} that    
 $\frac 12 c_F= \dd_F^*(\uu, \uu_\pm)$ and $P_{\mm_\pm} \uu=\uu_\pm.$
For any curve  $\bxi_\pm:[0,1]\mapsto \R^n$ with $\bxi_\pm(0)=\uu$ and $\bxi_\pm(1)=\uu_\pm$, by the co-area formula,  we have 
\begin{align*}
& \int_0^1 \sqrt{2F(\bxi(t))}|\bxi'(t))|\, dt\\
\overset{\eqref{bulk potential}}\geq & \int_0^1 \sqrt{2f(\dd_\mm^2(\bxi(t)))}\Big|\frac d {dt} \dd_\mm(\bxi(t))\Big|\, dt\\
=&  \int_{\dd_\mm(\uu_\pm)}^{\dd_\mm(\uu)} \sqrt{2f(\lambda^2)} \, d\lambda=\int_0^{|\uu-\uu_\pm|} \sqrt{2f(\lambda^2)} \, d\lambda.
\end{align*}
Note that the last step is due to $P_{\mm_\pm} \uu=\uu_\pm$.
Taking the infimum among all such curves we find 
\begin{align*}
\int_0^{|\uu-\uu_\pm|} \sqrt{2f(\lambda^2)} \, d\lambda\leq \dd_F^*(\uu,\uu_\pm)=\frac 12 c_F=
\int_{0}^{\frac {\dist_{\mm}}2} \sqrt{2f(\lambda^2)} d \lambda.
\end{align*}
 This   implies  that $|\uu_\pm-\uu|\leq \frac 12 \dist_\mm$, and thus  $|\uu_+-\uu_-|\leq \dist_\mm$ by  triangle inequality. This leads to  a  contradiction, and    the claim \eqref{ringnonoverlap} is proved.

%It follows from  \eqref{def normalsphere} that  $|\uu_\pm -\uu|\leq \frac 12 \dist_\mm$ and $\frac 12 c_F= \dd_F^*(\uu, \uu_\pm)$.
 %Consider the line $\bxi_\pm (t)=t\uu +(1-t)\uu_\pm$ for $t\in [0,1]$. We have 
 %\begin{align*}
 %\frac 12 c_F= \dd_F^*(\uu, \uu_\pm)\overset{\eqref{semidistance1}}\leq \int_0^1 \sqrt{2F(\bxi_\pm)} |\bxi_\pm'|\, dt\overset{\eqref{bulk potential}}%=\int_0^{|\uu_+-\uu|}\sqrt{2f(\lambda^2)}\,d\lambda\overset{\eqref{linpanwang cf equ}}\leq \frac 12 c_F.
 %\end{align*}
 %This would imply  that $|\uu_\pm-\uu|=\frac 12 \dist_\mm$, and thus  triangle inequality implies  $|\uu_+-\uu_-|\leq \dist_\mm$. This would  contradict to the assumption that  $(\uu_+,\uu_-) $ is NOT  a minimal pair, and  thus  the claim \eqref{ringnonoverlap} is proved. 
 
Using  \eqref{ringnonoverlap} and  the continuity of $ \dd_F^*$, we deduce
  $$ \dd_F^*\(\mathcal{N}(\uu_+,\tfrac 12 c_F), \mathcal{N}(\uu_-,\tfrac 12 c_F)\)=:\beta>0.$$ This combined with  the  triangle inequality yields 
\begin{align*}
& \dd_F^*\(\mathcal{N}(\uu_+,\delta_0), \mathcal{N}(\uu_-,\delta_0)\)\\
\leq &~ \dd_F^*\(\mathcal{N}(\uu_+,\delta_0), \mathcal{N}(\uu_+,\tfrac 12 c_F)\)\\
&~+ \dd_F^*\(\mathcal{N}(\uu_+,\tfrac 12 c_F), \mathcal{N}(\uu_-,\tfrac 12 c_F)\)\\&~+ \dd_F^*\(\mathcal{N}(\uu_-,\tfrac 12 c_F), \mathcal{N}(\uu_-,\delta_0)\)\\
=& ~c_F-2\delta_0+\beta.
\end{align*}
This implies $\kappa(\uu_+,\uu_-,0)=\beta>0$.

% and the minimizer is achieved by the line segment  . By Cauchy-Schwarz's inequality and constantly extending  such a trajectory to $\R$, we obtain a minimizer to $\cc_F(\uu_+,\uu_-)$ (cf. \eqref{der mc3}). This combined with Lemma \ref{lemma mc} yields that $(\uu_+,\uu_-)$ is a minimal pair,  contradicting to the assumption. 

%By the continuity of $ \dd_F^*(\cdot,\cdot)>0$,  there exists $\rho(\uu_+,\uu_-)$ such that 
%\[\mathcal{J}\(B_{\rho}(\uu_+)\cap \mm_+,\tfrac 12 c_F\)\cap \mathcal{J}\(B_{\rho}(\uu_-)\cap \mm_-,\tfrac 12 c_F\)=\varnothing.\]
\end{proof}
\begin{lemma}\label{wanglemma2}
There exist  constants $C_0=C_0(\mm)$ and $\rho_0=\rho_0(\mm)$ which only depend on the geometry of $\mm$ so that the following holds: 
 \[ \text{for any curve }\bgamma\in H^1([-\delta,\delta], \R^n) \text{ with } \gamma(\pm \delta)\in B_{\delta_0}(\mm_\pm),\label{conditioncurve}\] and any $\rho\in (0,\rho_0)$, there holds
\[\int_{-\delta}^{\delta} \frac 1 2|  \bgamma'   |^2 +\frac 1{\e^2} F (\bgamma    )-\frac 1\e   \(\dd_F\circ \bgamma\)'  \geq \frac {\min\left\{C_0\rho^2,\kappa\Big(P_{\mm} \bgamma(\delta),P_{\mm} \bgamma(-\delta),\rho\Big)\right\}}{\max\{\e,\delta\}}.\]
\end{lemma}
\begin{proof}
For any  curve $\bgamma$ satisfying \eqref{conditioncurve}, we define its first  exit time of  $B_{\delta_0}(\mm_-)$ and last  entrance  time of $B_{\delta_0}(\mm_+)$ respectively by 
\begin{align*}
t_-&=\inf\{ t\in (-\delta,\delta): \bgamma(s)\in B_{\delta_0}(\mm_-)~\text{ for }~s\in (-\delta,t)\},\\
t_+&=\sup\{ t\in (-\delta,\delta): \bgamma(s)\in B_{\delta_0}(\mm_+)~\text{ for } ~s\in (t,\delta)\}.
\end{align*}
We shall estimate three integrals
\[I_-+I_0+I_+:= \(\int_{-\delta}^{t_-}+\int_{t_-}^{t_+}+\int_{t_+}^{\delta}\) \(\frac 1 2|  \bgamma'   |^2 +\frac 1{\e^2} F (\bgamma)-\frac 1\e   \(\dd_F\circ \bgamma\)'\).\]
For $s\in (-\delta, t_-)\cup (t_+,\delta)$, we have $\bgamma(s)\in B_{\delta_0}(\mm)$. 
So we can define the normal projection of $\bgamma'$ by 
$\Pi_{\bgamma}\bgamma' =\bgamma' \cdot \bnu_\mm(\bgamma)  \bnu_\mm(\bgamma)$. Recall that   $\bnu_\mm(\bgamma)~\parallel~ \p \dd_F (\bgamma)$ (cf.  \eqref{normal of m} and \eqref{normalize pdf}). By a similar calculation as \eqref{linear control nonlinar}, we obtain
\[|\bgamma'  -\Pi_{\bgamma}\bgamma'|^2 
 \geq    (1-C\delta_0)|  (P_{\mm} \bgamma)'|^2\quad \text{ on } \quad (-\delta, t_-)\cup (t_+,\delta).\label{localproest}\] 
 Here $C$ is a constant depending on the geometry of $\mm$.
 Choosing $\delta_0\ll 1$ so that $1-C\delta_0=: 2c_0>0$, we find  
\begin{align*}
I_+ 
&\overset{\eqref{quasidistance}}=\frac 1 2 \int_{t_+}^{\delta} \(|  \bgamma'   |^2 +\frac 1{ \e^2} |\p \dd_F (\bgamma)|^2-\frac 2\e   \p\dd_F(\bgamma)\cdot  \bgamma ' \) \\
& =\frac 1 2 \int_{t_+}^{\delta} |  \bgamma'  -\Pi_{\bgamma}\bgamma'|^2+\Big|\Pi_{\bgamma}\bgamma'-\tfrac 1\e \p \dd_F(\bgamma)\Big |^2 \\
&\overset{\eqref{localproest}} \geq    c_0\int_{t_+}^\delta  |  (P_{\mm} \bgamma)'|^2\\
& \geq     \frac{c_0}{  \delta-t_+ } \Big|P_{\mm}\bgamma(t_+)-P_{\mm}\bgamma(\delta )\Big|^2.
\end{align*}
The same  calculation of  $I_-$ leads to 
$  I_-\geq  \frac{c_0}{  t_-+\delta } |P_{\mm}\bgamma(t_-)-P_{\mm}\bgamma(-\delta)|^2.$   

For $\rho\in (0,\rho_0)$ with $\rho_0=\rho_0(\mm)$ being sufficiently small, we have at least one of the following  three cases:
 \begin{enumerate}
 \item If  $\bgamma(t_+)\notin \mathcal{N}\Big(\bar{B}_\rho(P_{\mm}\bgamma(\delta))\cap\mm_+,\delta_0\Big)$, then $|P_\mm\bgamma(t_+)-P_{\mm}\bgamma(\delta)|\geq C_1\rho$ for some $C_1$ that only depends on the geometry of $\mm$. Thus $I_+\geq \frac{C \rho^2}{\delta- t_+ }$. 
\item  If  $\bgamma(t_-)\notin \mathcal{N}\Big(\bar{B}_\rho(P_{\mm}\bgamma(-\delta))\cap\mm_-,\delta_0\Big)$, then   $I_-\geq \frac{C \rho^2}{t_-+\delta }$.
\item If neither of the above cases happen, i.e.  $\bgamma(t_\pm)\in \mathcal{N}\Big(\bar{B}_\rho(P_{\mm}\bgamma(\pm\delta))\cap\mm_\pm,\delta_0\Big)$, then it follows from Cauchy-Schwarz's inequality, \eqref{semidistance1} and \eqref{quasidistance} that 
 \begin{align*}
I_0& =  \frac 1\e \int_{t_-}^{t_+}  |  \bgamma'   |\sqrt{2 F (\bgamma)}\,d\tau-\frac {\dd_F\circ\bgamma(t_+)-\dd_F\circ\bgamma(t_-)}\e \\
& \geq  \frac 1\e  \dd_F^* (\bgamma(t_+),\bgamma(t_-)) -\frac {c_F-2\delta_0}\e     \\
& \geq  \frac 1\e  \dd_F^* \Big(\mathcal{N}\(\bar{B}_{\rho}(P_{\mm}\bgamma(\delta))\cap \mm_+,\delta_0\),\mathcal{N}\(\bar{B}_{\rho}(P_{\mm}\bgamma(-\delta))\cap \mm_-,\delta_0\)\Big) -\frac {c_F-2\delta_0}\e     \\
&\overset{\eqref{lowerestI0}}= \frac 1\e  \kappa\Big(P_{\mm}\bgamma(\delta),P_{\mm}\bgamma(-\delta),\rho\Big).
 \end{align*}
 \end{enumerate}
 
\end{proof}

 \begin{proof}[Proof of  \eqref{thm minimal pair}]
We shall argue for every $t\in [0,T]$ without mentioning each time  in the sequel.
We first  use
   \eqref{deri con1} to deduce      strong convergence of $ \uu_{\e_k}$ on almost every slices. More precisely, there exists a null set $N\subset [0,\delta_0]$, namely   $\mathcal{L}^1(N)=0$,  so that   
   \[\uu_{\e_k} \left(p\pm  \delta  \nn (p)\right)\xrightarrow{k\to\infty} \uu^\pm(p\pm  \delta  \nn (p))\text{ strongly    in }  L^2(\Sigma_t)\qquad \forall \delta\notin N.\label{uniform convergence slice4}\]
 Here $\nn(p)$ is the normal vector of $p\in\Sigma_t$.
 Note that the limit in \eqref{uniform convergence slice4} makes sense due to 
  \eqref{upm regu} and Sobolev's trace theorem. Moreover,     
     \[  \uu^\pm  \left(p\pm  \delta  \nn (p)\right)\rightarrow \uu^\pm   (p ) \text{ strongly in } L^2(\Sigma_t)\text{ as } \delta\downarrow  0,\delta\notin N.\label{uniform convergence slice2}\]
   Combining \eqref{uniform convergence slice4} with \eqref{uniform convergence slice2} and a diagonal argument, we find a sequence $\delta_k\downarrow 0$  so that 
  \[\uu_{\e_k} \left(p\pm  \delta_k  \nn (p)\right)\xrightarrow{k\to\infty} \uu^\pm(p)\text{ strongly    in }  L^2(\Sigma_t).\label{uniform convergence slice1}\]
It follows from \eqref{calibration est2}, \eqref{def:xi} and the orthogonal decomposition $\nabla=\nabla_\Sigma+\nn\p_\nn$ that 
\begin{align*}
 \sup_{t\in [0,T]} \int_{B_{\delta_0}(\Sigma_t)}  \( \frac{\e}2  |\nabla_\Sigma \uu_{\e_k}  | ^2+\frac{\e}2  |\p_\nn \uu_{\e_k}  | ^2 +\frac1 {\e}{F (\uu_{\e_k} )} -\bxi\cdot \nn \p_\nn \psi_\e  \)\, d x \leq C\e.
\end{align*}
Owning to 
 \eqref{phi func control} and \eqref{energy bound3}, we have 
\begin{align*}
&\sup_{t\in [0,T]} \int_{B_{\delta_0}(\Sigma_t)}  \(  \frac 12  |\p_\nn \uu_{\e_k}  | ^2 +\frac1 {\e^2}{F (\uu_{\e_k} )} - \frac1 \e \p_\nn \psi_{\e_k}  \)\, d x\nonumber\\
 \leq &~ C+\e^{-1}\sup_{t\in [0,T]} \int_{B_{\delta_0}(\Sigma_t)}  \(1-\phi\(\frac{\dd_\Sigma}{\delta_0}\)\) |\nabla\psi_{\e_k}|\, dx \\
 \leq &~ C+ C\e^{-1}\sup_{t\in [0,T]} \int_{B_{\delta_0}(\Sigma_t)}  \min (\dd_\Sigma^2,1)  |\nabla\psi_{\e_k}|\, dx\\
 \leq &~ C\(1+\e^{-1}\sup_{t\in [0,T]} E_\e  [ \uu_{\e_k}   | \Sigma ]\).
\end{align*}
If we write $\uu_{\e_k}(x,t)=\uu_{\e_k}(p+  \tau \nn (p),t)$, then  by the area formula, we deduce that 
\begin{align}\label{normalexpand1}
&\sup_{t\in [0,T]} \int_{\Sigma_t} \underbrace{\int_{-\delta_k}^{\delta_k}\left(\frac 1 2|\p_\tau \uu_{\e_k}   |^2 +\frac 1{\e^2} F\left(\uu_{\e_k}   \right)-\frac 1\e \p_\tau \(\dd_F\circ \uu_{\e_k}\) \right)  \,d \tau}_{=:\Theta_k (p,t)}  dS(p)\leq C.
\end{align}
For each fixed $p\in \Sigma_t$ and $k$, we consider   the curves 
 \[\gamma_k(\tau;p) :=\uu_{\e_k}(p+  \tau \nn (p)): [-\delta_k,\delta_k]\mapsto \R^n.\label{gammacurve1}\]

 Assume that  there exists $t\in [0,T]$ and  a compact subset $E_t^*\subset \Sigma_t$ with $\mathcal{H}^{d-1}(E_t^*)\geq \alpha $ for some $\alpha\in (0,1/2)$ s.t. the pairs $(\uu^+(p), \uu^-(p))_{p\in E_t^*}\subset \mm_+\times \mm_-$, obtained by taking the one-sided trace of  \eqref{upm regu},  are not   minimal pairs. Then it follows from    Lemma \ref{wanglemma1}     that   
 \[\kappa(\uu^+(p),\uu^-(p),0)>0,\quad \forall p\in E_t^*.\label{kappapositive1}\]
 By  \eqref{uniform convergence slice1} and  Egorov's  theorem, there exists    a compact subset $E_t\subset E_t^*$ with $\mathcal{H}^{d-1}(E_t)\geq 0.9 \mathcal{H}^{d-1}(E_t^*)\geq  0.9\alpha $ s.t.
  \[\gamma_k(\pm \delta_k;p)= \uu_{\e_k} \left(p\pm  \delta_k  \nn (p)\right)\xrightarrow{k\to\infty} \uu^\pm(p)\text{ uniformly    on  }   E_t .\label{uniform convergence slice3}\]
As a result,  $\uu^\pm(\cdot)$ are  continuous on $E_t$ because  $\uu_{\e_k} \left(p\pm  \delta_k  \nn (p)\right)$ are continuous on $\Sigma_t$. As a result,  there exists $K>0$ so that 
\[ \gamma_k(\pm \delta_k;p)=\uu_{\e_k} \left(p\pm  \delta_k  \nn (p)\right)\in B_{\delta_0}(\mm_\pm)\qquad \forall k\geq K,p\in E_t.\label{endpointcurve}\]
 The continuity of $\uu^\pm(\cdot)$ and   \eqref{lowerestI0} imply that  
 $$\kappa(\uu^+(\cdot),\uu^-(\cdot),\cdot):  E_t\times [0,1]\mapsto [0,\infty)\text{ is continuous}.$$  
So we deduce from  \eqref{kappapositive1} that   there exists $\rho\in (0,\rho_0)$ s.t. 
 \[\inf_{p\in E_t}\kappa(\uu^+(p),\uu^-(p),\rho)=:2\beta>0.\]
 Note that $\beta$ is independent of $k$.
Owning to \eqref{uniform convergence slice3}, the  continuities  of $\kappa$ and that of  the projection $P_\mm$, there exists $K>0$ s.t.
 \[\inf_{p\in E_t}\kappa\Big(P_\mm\gamma_k( \delta_k;p),P_\mm\gamma_k(- \delta_k;p),\rho\Big)\geq \beta,\qquad \forall k\geq K.\]
 By \eqref{endpointcurve},  the curves $\gamma_k(~\cdot~;p)$ satisfy   the condition \eqref{conditioncurve}  so that  
 Lemma \ref{wanglemma2} applies to the 1D integral $\Theta_k$ defined in \eqref{normalexpand1}:
\begin{align*}
\Theta_k (p,t)&\geq \frac 1 {\max\{\e_k,\delta_k\}}\min\Big\{C_0\rho^2,\kappa\Big(P_\mm\gamma_k( \delta_k;p),P_\mm\gamma_k(- \delta_k;p),\rho\Big)\Big\}\\
&\geq \frac {\min\{C_0\rho^2,\beta\}} {\max\{\e_k,\delta_k\}} ,\qquad\qquad\qquad ~\forall p\in E_t,k\geq K.
\end{align*}
 However, this would contradict \eqref{normalexpand1} and we thus finish the proof of   \eqref{thm minimal pair}.
\end{proof}

%%%%%%%%%%%%%%%%%%%%%%%%%%%%%%%%%%%%%%%%%%%%%%%%%%%%%%%%%%%%%%%%%%%%%%%%%%%%%%%%%%%%%%%%%%%%%%%%%%%%%%%%%%%%
   \section{Proof of Theorem \ref{thm init}: construction of initial data}\label{sec initial data}
We shall first modify and extend $\uu^{in}_\pm$ in the transitional region $B_{2\delta}(\Sigma_0)$ so that we can glue them into a new mapping that fulfills the desired properties in Theorem \ref{thm init}. To this aim, let $\Psi_\delta^\pm: \O_0^\pm\cup  \bar{B}_\delta(\Sigma_0)\mapsto \O_0^\pm$ be  global $C^1$ diffeomorphisms up to the boundaries which result from gluing the identity mapping in $ \O_0^\pm \backslash B_{2\delta}(\Sigma_0)$ and the projection mapping $P_{\Sigma_0}$  in $\bar{B}_{\delta}(\Sigma_0)$:
\begin{align}\label{psi fix}
\Psi_\delta^\pm(x)=\left\{
\begin{array}{rl}
P_{\Sigma_0}(x) &\text{ in }\bar{B}_{\delta}(\Sigma_0),\\
x &\text{ in   }  \O_0^\pm \backslash B_{2\delta}(\Sigma_0).
\end{array}
\right.
\end{align}
We    extend  $\uu^{in}_\pm$  by defining   
\[\uu_0^{\pm}:= 
\uu^{in}_\pm\circ \Psi_\delta^\pm \in H^1(\O_0^\pm\cup \bar{B}_{\delta}(\Sigma_0),\mm_\pm). \label{u in extension}
 \]
 This combined with \eqref{MC initial data} implies that 
 \begin{align}
(\uu_0^+(x),\uu_0^-(x))|_{ x\in B_\delta(\Sigma_0)}\text{ are  minimal pairs}.\label{initial mp}
\end{align}
We shall construct $\uu_\e^{in}$ by gluing $\uu_0^\pm$. To this end,  we define  a cut-off function 
\[\eta_\delta\in C_c^\infty(B_\delta(\Sigma_0),[0,1])~\text{ with }\eta_\delta=1\text{  in }B_{\delta/2}(\Sigma_0).\label{cut-off eta delta2}\] 
Recall  the optimal profile  $\alpha$  \eqref{optimal profile alpha}.
We define 
\[S_\e(x)=\eta_\delta(x) \alpha\(\tfrac{\dd_\Sigma(x,0)}\e\)+(1-\eta_\delta(x))\tfrac {\dist_\mm}2 \(\1_{\O^+_0}(x)-\1_{\O^-_0}(x)\).\label{se def}\]
It is easy to verify that $S_\e$ is a smooth function (as   the discontinuity caused by $\1_{\O^\pm_0}$ is cut off by $\eta_\delta$). We write
\[S_\e(x)= \alpha\(\tfrac {\dd_\Sigma(x,0)}\e\)-\hat{S}_\e(x),\]
where $\hat{S}_\e$ is the tail term 
\[\hat{S}_\e(x)=(1-\eta_\delta(x))\left[\alpha\(\tfrac {\dd_\Sigma(x,0)}\e\)-\tfrac {\dist_\mm}2 \(\1_{\O^+_0}(x)-\1_{\O^-_0}(x)\)\right].\label{hat s def}\]
By Rademacher's theorem,
$\dd_\Sigma(x,0)$ is Lipschitz continuous in  $\O$,  and   $|\nabla \dd_\Sigma(x,0)|\leq 1$ a.e. in $\O$. This combined with   \eqref{exp alpha}  
yields
\[\|\hat{S}_\e\|_{L^{\infty}(\O)}+\|\nabla \hat{S}_\e\|_{L^{\infty}(\O)}\leq C e^{-C/\e}.\label{exp hatS}\]
and thus
\begin{subequations}
\begin{align}
S_\e(x)&= \alpha\(\tfrac {\dd_\Sigma(x,0)}\e\)+O(e^{-C/\e})\text{ in } \O,\label{def Se1}\\
\nabla S_\e(x)&= \tfrac {\nabla \dd_\Sigma(x,0)}\e\alpha'+O(e^{-C/\e})\quad \text{a.e.  in } \O,\label{def Se2}\\
S_\e(x)&\xrightarrow{\e\to 0} \tfrac {\dist_\mm}2 \(\1_{\O^+_0}(x)-\1_{\O^-_0}(x)\)\text{ a.e.  in }\O.
\end{align}
\end{subequations}
Using \eqref{u in extension} and \eqref{se def}, we define $\uu_\e^{in}$ by 
\begin{align}\label{uu initial}
\uu_\e^{in}(x)=\tfrac {\uu_0^+(x)+\uu_0^-(x)}2 +S_\e(x)\tfrac{\uu_0^+(x)-\uu_0^-(x)}{\dist_\mm}.
\end{align}
  We claim  \eqref{u coincide} holds. Indeed in the domains $\O_0^\pm\backslash B_{2\delta}(\Sigma_0)$, we have $\eta_\delta=0$ and thus
  \[\uu_\e^{in}\overset{\eqref{se def}}=\tfrac {\uu_0^+ +\uu_0^-}2+\tfrac {\dist_\mm}2 (\1_{\O^+_0} -\1_{\O^-_0})\tfrac{\uu_0^+-\uu_0^-}{\dist_\mm}=\sum_\pm \uu_0^\pm\1_{\O^\pm_0}\quad \text{ in } \O_0^\pm\backslash B_{2\delta}(\Sigma_0). \]
  This combined with   \eqref{psi fix} yields
   \[\uu_\e^{in}(x)  =\uu^{in}_\pm\circ \Psi_\delta^\pm(x) =\uu^{in}_\pm(x),\quad \forall x\in \O_0^\pm\backslash B_{2\delta}(\Sigma_0),\]
   and \eqref{u coincide} is proved.
   
Now we turn to the proof  of  \eqref{u cali}.
Substituting   \eqref{def Se1} and \eqref{def Se2} into \eqref{uu initial}, we obtain 
\begin{align}\label{grad uuin}
\nabla\uu_\e^{in}&=\tfrac {\nabla\uu_0^++\nabla\uu_0^-}2 +\tfrac{\nabla\dd_\Sigma}\e\alpha'(\tfrac{\dd_\Sigma}{\e})  \tfrac {\uu_0^+-\uu_0^-}{\dist_\mm}\nonumber\\
&+\alpha(\tfrac{\dd_\Sigma}{\e}) \tfrac {\nabla\uu_0^+ -\nabla\uu_0^-}{\dist_\mm}+ O(e^{-C/\e})\(|\nabla\uu_0^+|+|\nabla\uu_0^-|+1\)\qquad \text{a.e. in }\O.
\end{align}
 
\subsection{Proof of \eqref{u cali}: Estimate near $\Sigma_0$}
By \eqref{initial mp}  
\[(\uu_0^+(x)-\uu_0^-(x))\perp_{\R^n} T_{\uu_0^\pm(x)}\mm_\pm,\quad \forall x\in B_\delta(\Sigma_0).\] As $\uu_0^\pm$ are mappings  into $\mm_\pm$ respectively, we have $\nabla \uu_0^\pm(x)\in T_{\uu^\pm(x)}\mm_\pm$. So   
\[(\uu_0^+(x)-\uu_0^-(x))\cdot \p_{x_i} \uu_0^\pm(x),\quad a.e. ~ x\in B_\delta(\Sigma_0),\quad 1\leq i\leq d.\label{mc orth gradi}\] 
The square of the second term on the right-hand side of \eqref{grad uuin} is
\[\left|\tfrac{\nabla \dd_\Sigma}\e\alpha'(\tfrac{\dd_\Sigma}{\e})  \tfrac {\uu_0^+-\uu_0^-}{\dist_\mm}\right|^2\overset{\eqref{initial mp}}=\e^{-2} \(\alpha'(\tfrac{\dd_\Sigma}{\e})\)^2\text{ in }B_{\delta}(\Sigma_0).\]
This together with  \eqref{mc orth gradi} enables us to compute  the square of \eqref{grad uuin} by 
 \begin{align}\label{grad uuin1}
|\nabla\uu_\e^{in}|^2=&  \e^{-2} \(\alpha'(\tfrac{\dd_\Sigma}{\e})\)^2   +\left|\tfrac{\nabla\uu_0^++\nabla\uu_0^-}2+\alpha(\tfrac{\dd_\Sigma}{\e}) \tfrac{\nabla\uu_0^+ -\nabla\uu_0^-}{\dist_\mm}\right|^2\nonumber\\
&+ O(e^{-C/\e})\(|\nabla\uu_0^+|^2+|\nabla\uu_0^-|^2+1\)\qquad \text{ in }B_{\delta}(\Sigma_0).
\end{align}
By \eqref{def Se1} and \eqref{uu initial}, we have 
\[F(\uu_\e^{in})=F\(\tfrac {\uu_0^++\uu_0^-}2 +\alpha(\tfrac{\dd_\Sigma}{\e})\tfrac {\uu_0^+-\uu_0^-}{\dist_\mm}\)+O(e^{-C/\e})\qquad \text{ in } \O.\label{F init expand1}\]
To compute  the RHS of \eqref{F init expand1}, we first deduce from \eqref{optimal profile alpha} that 
\[\bgamma(s):\R\mapsto \tfrac {\uu_0^++\uu_0^-}2 +\alpha\(s\)\tfrac {\uu_0^+-\uu_0^-}{\dist_\mm}\]
parametrizes   the line segment $\overline{\uu_0^-\uu_0^+}$ with  $\gamma(0)=\tfrac {\uu_0^++\uu_0^-}2$  the middle point.
For $x\in B_\delta(\Sigma_0)\cap \O_0^+$, we have $\alpha(\tfrac{\dd_\Sigma(x)}{\e})>0$ and by \eqref{odd increase}, the distance from $\bgamma(\frac{\dd_\Sigma(x)}\e)$ to $\mm$ equals to  the distance  between $\bgamma(\frac{\dd_\Sigma(x)}\e)$ and $\uu_0^+\in\mm_+$. So
\[\left|\dd_\mm\(\tfrac {\uu_0^++\uu_0^-}2 +\alpha(\tfrac{\dd_\Sigma}{\e})\tfrac {\uu_0^+-\uu_0^-}{\dist_\mm}\)\right|=\left|\tfrac{\dist_\mm}2-\alpha(\tfrac{\dd_\Sigma}{\e})\right|  \left| \tfrac {\uu_0^+-\uu_0^-}{\dist_\mm}\right|
 \overset{\eqref{initial mp}}=\left|\tfrac{\dist_\mm}2-\alpha(\tfrac{\dd_\Sigma}{\e})\right| \label{dist repre}\]
on  $ B_\delta(\Sigma_0)\cap \O_0^+$, and thus
\begin{align}\label{dist repre1}
&F\(\tfrac {\uu_0^++\uu_0^-}2 +\alpha(\tfrac{\dd_\Sigma}{\e})\tfrac {\uu_0^+-\uu_0^-}{\dist_\mm}\)\nonumber\\
\overset{\eqref{bulk potential},\eqref{dist repre}}=&f\(\left|\tfrac{\dist_\mm}2-\alpha(\tfrac{\dd_\Sigma}{\e})\right|^2\)
\overset{\eqref{centralized  potential}}=\tilde{F}\(\alpha(\tfrac{\dd_\Sigma}{\e})\).
\end{align}
Similar calculation leads to  the case when $x\in B_\delta(\Sigma_0)\cap \O_0^-$, and altogether   we have (from \eqref{F init expand1} and \eqref{dist repre1})  that 
\[F(\uu_\e^{in})=\tilde{F}\(\alpha(\tfrac{\dd_\Sigma}{\e})\)+O(e^{-C/\e})\qquad \text{ in } B_\delta(\Sigma_t).\label{F init expand2}\]
Now we   compute  $\bxi\cdot\nabla( \dd_F \circ \uu_\e^{in})$:
\begin{align}\label{F init expand3}
 -\int_\O \eta_\delta \bxi\cdot\nabla(\dd_F \circ \uu_\e^{in}) &\overset{\eqref{cut-off eta delta2}}= \int_\O \div (\eta_\delta \bxi ) \dd_F \circ \uu_\e^{in} \nonumber\\&
 \overset{\eqref{uu initial}}=  \int_\O \div (\eta_\delta\bxi ) \dd_F  \(\tfrac{\uu_0^+ +\uu_0^- }2 +S_\e \tfrac{\uu_0^+ -\uu_0^- }{\dist_\mm}\).
\end{align}
Using \eqref{def Se1} and the Lipschitz property of $\dd_F$, we can write 
\begin{align}\label{F init expand4}
-\int_\O \eta_\delta\bxi\cdot\nabla(\dd_F \circ \uu_\e^{in}) =& \int_\O  \div( \eta_\delta \bxi ) \dd_F  \(\tfrac{\uu_0^+ +\uu_0^- }2 +\alpha(\tfrac{\dd_\Sigma}{\e})\tfrac{\uu_0^+ -\uu_0^- }{\dist_\mm}\) +O(e^{-C/\e}).
\end{align}
To   evaluate   $\dd_F$ on  $B_\delta(\Sigma_t)$, we note that    a line segment inside a minimal connection will not  involve the 2nd case  defining $\dd_F$ (cf.   \eqref{quasidistance}). As a result, at  any $ x\in B_\delta(\Sigma_0)\cap \O_0^+$, since $\alpha(\tfrac{\dd_\Sigma}{\e})>0$ and $(\uu_0^+, \uu_0^-)$ is a   minimal pair (cf. \eqref{initial mp}), we have
\begin{align*}
  &\dd_F  \(\tfrac{\uu_0^+ +\uu_0^-}2 +\alpha(\tfrac{\dd_\Sigma}{\e})\tfrac{\uu_0^+ -\uu_0^-}{\dist_\mm}\)\\
 & \overset{\eqref{quasidistance}}=c_F-\int_0^{\dd_{\mm_+}\big(\tfrac{\uu_0^+ +\uu_0^-}2 +\alpha(\tfrac{\dd_\Sigma}{\e})\tfrac{\uu_0^+ -\uu_0^-}{\dist_\mm}\big)}\sqrt{2f(\lambda^2)}\, d\lambda\\
 & \overset{\eqref{dist repre}}=c_F-\int_0^{\frac{\dist_\mm}2-\alpha(\tfrac{\dd_\Sigma}{\e})}\sqrt{2f(\lambda^2)}\, d\lambda\\
& \overset{\eqref{centralized  potential}}=c_F-\int_{\alpha(\tfrac{\dd_\Sigma}{\e})}^{\frac{\dist_\mm}2}\sqrt{2\tilde{F}(\lambda)}\, d\lambda\qquad \forall x\in B_\delta(\Sigma_0)\cap \O_0^+.
\end{align*}
Similar calculation applies to  $  B_\delta(\Sigma_0)\cap \O_0^-$ and $\alpha(\tfrac{\dd_\Sigma}{\e})<0$:
\begin{align*}
  &\dd_F  \(\tfrac{\uu_0^+ +\uu_0^-}2 +\alpha(\tfrac{\dd_\Sigma}{\e})\tfrac{\uu_0^+ -\uu_0^-}{\dist_\mm}\)\\
 & \overset{\eqref{quasidistance}}= \int_0^{\dd_{\mm_-}\(\tfrac{\uu_0^+ +\uu_0^-}2 +\alpha(\tfrac{\dd_\Sigma}{\e})\tfrac{\uu_0^+ -\uu_0^-}{\dist_\mm}\)}\sqrt{2f(\lambda^2)}\, d\lambda\\
 & \overset{\eqref{linpanwang cf equ}}=\int_0^{\frac{\dist_\mm}2+\alpha(\tfrac{\dd_\Sigma}{\e})}\sqrt{2f(\lambda^2)}\, d\lambda\\
 & \overset{\eqref{centralized  potential}}=\int_{-\frac{\dist_\mm}2}^{\alpha(\tfrac{\dd_\Sigma}{\e})}\sqrt{2\tilde{F}(\lambda)}\, d\lambda\qquad  \forall x\in B_\delta(\Sigma_0)\cap \O_0^-.
\end{align*}
To summarize, we have
\[\label{longcase1}
  \dd_F  \(\tfrac{\uu_0^+ +\uu_0^-}2 +\alpha(\tfrac{\dd_\Sigma}{\e})\tfrac{\uu_0^+ -\uu_0^-}{\dist_\mm}\)=\left\{
   \begin{split}
  c_F- \int_{\alpha(\tfrac{\dd_\Sigma}{\e})}^{\frac{\dist_\mm}2}\sqrt{2\tilde{F}(\lambda)}\, d\lambda\qquad \forall x\in B_\delta(\Sigma_0)\cap \O_0^+,\\
   \int_{-\frac{\dist_\mm}2}^{\alpha(\tfrac{\dd_\Sigma}{\e})}\sqrt{2\tilde{F}(\lambda)}\, d\lambda
\qquad \forall x\in B_\delta(\Sigma_0)\cap \O_0^-.
   \end{split}
   \right.
\]
Recall from \eqref{cut-off eta delta2} that $\eta_\delta$ vanishes outside $B_\delta(\Sigma_0)$. So substituting \eqref{longcase1} into \eqref{F init expand4} and integrating by parts
 yield
 \begin{align}\label{F init expand5} 
\int_\O \eta_\delta\bxi\cdot\nabla(\dd_F \circ \uu_\e^{in}) =& \int_\O  \e^{-1} \eta_\delta\bxi \cdot \nabla\dd_\Sigma \alpha'(\tfrac{\dd_\Sigma}{\e}) \sqrt{2\tilde{F}\(\alpha(\tfrac{\dd_\Sigma}{\e})\)}\, dx +O(e^{-C/\e}).
\end{align}
This combined with \eqref{F init expand2} and \eqref{grad uuin1} yields
\begin{align}\label{grad uuin2}
&\int_\O \(\frac 12 |\nabla\uu_\e^{in}|^2+\frac{F(\uu_\e^{in})}{\e^2}- \frac 1 \e \bxi\cdot\nabla(\dd_F \circ \uu_\e^{in})\)\eta_\delta\nonumber\\
&= \int_\O  \(\e^{-2} \frac 12 \(\alpha'(\tfrac{\dd_\Sigma}{\e})\)^2 +\e^{-2}\tilde{F}\(\alpha(\tfrac{\dd_\Sigma}{\e})\) -\e^{-2} \bxi \cdot \nabla\dd_\Sigma \alpha'(\tfrac{\dd_\Sigma}{\e}) \sqrt{2\tilde{F}\(\alpha(\tfrac{\dd_\Sigma}{\e})\)} \)\eta_\delta\nonumber\\
&  +\int_\O \frac 12  \eta_\delta \left|\tfrac{\nabla\uu_0^++\nabla\uu_0^-}2+\alpha(\tfrac{\dd_\Sigma}{\e}) \tfrac{\nabla\uu_0^+ -\nabla\uu_0^-}{\dist_\mm}\right|^2\nonumber\\
&+O(e^{-C/\e})\int_\O   \(1+ |\nabla\uu_0^+|^2+|\nabla\uu_0^-|^2\)\eta_\delta.\end{align}
By \eqref{alpha'=} the  first term on the right-hand side above simplifies to 
\begin{align}
\int_\O  \eta_\delta\e^{-2}  (1- \bxi \cdot \nabla\dd_\Sigma) \(\alpha'(\tfrac{\dd_\Sigma}{\e})\)^2\overset{\eqref{def:xi}}=\int_\O O(1)\eta_\delta  \tfrac{\dd_\Sigma^2}{\e^2} \(\alpha'(\tfrac{\dd_\Sigma}{\e})\)^2\overset{\eqref{exp alpha}}=O(\e).
\end{align}
The above two equations together implies
\begin{align}\label{grad uuin3}
&\int_\O \(\frac 12 |\nabla \uu_\e^{in}|^2+\frac{F(\uu_\e^{in})}{\e^2}- \frac 1 \e \bxi\cdot\nabla(\dd_F \circ \uu_\e^{in})\)\eta_\delta\nonumber\\
&=\int_\O\frac 12  \eta_\delta \left|\tfrac{\nabla\uu_0^++\nabla\uu_0^-}2+\alpha(\tfrac{\dd_\Sigma}{\e}) \tfrac{\nabla\uu_0^+ -\nabla\uu_0^-}{\dist_\mm}\right|^2+O(\e)\int_\O   \(1+ |\nabla\uu_0^+|^2+|\nabla\uu_0^-|^2\)\eta_\delta.\end{align}

\subsection{Proof of \eqref{u cali}: Estimates away from $\Sigma_0$.}
 Using \eqref{exp alpha}, we have 
\[|\alpha'(\tfrac{\dd_\Sigma}{\e})|+\left|\alpha(\tfrac{\dd_\Sigma}{\e})- \tfrac{\dist_\mm}2 (\1_{\O_0^+}-\1_{\O_0^-})  \right|\leq Ce^{-C/\e}\text{   in }\O\backslash B_{\delta/2}(\Sigma_0).\label{exp alpha1}\]  Applying  the above estimates to \eqref{grad uuin} yields
\begin{align}\label{grad uuin4}
\nabla\uu_\e^{in} &= \tfrac {\nabla\uu_0^++\nabla\uu_0^-}2   +\tfrac{\dist_\mm}2 (\1_{\O_0^+}-\1_{\O_0^-})\tfrac {\nabla\uu_0^+ -\nabla\uu_0^-}{\dist_\mm}\nonumber\\
&\qquad  + O(e^{-C/\e})\(1+|\nabla\uu_0^+|\1_{\O_0^+}+|\nabla\uu_0^-|\1_{\O_0^-}\)\nonumber\\
&=\nabla\uu_0^+ \1_{\O_0^+}+\nabla\uu_0^- \1_{\O_0^-}\nonumber\\
&\qquad   + O(e^{-C/\e})\(1+|\nabla\uu_0^+|\1_{\O_0^+}+|\nabla\uu_0^-|\1_{\O_0^-}\) 
\quad \text{a.e. in }\O\backslash B_{\delta/2}(\Sigma_0).
\end{align}
By \eqref{cut-off eta delta2}
 the function  $(1-\eta_\delta)$ vanishes on $B_{\delta/2}(\Sigma_0).$ So multiplying this function to \eqref{grad uuin4} yields 
 \begin{align}\label{grad uuin5}
(1-\eta_\delta)|\nabla\uu_\e^{in}|^2 &=(1-\eta_\delta) \sum_\pm |\nabla\uu_0^\pm |^2\1_{\O_0^\pm}   \nonumber\\
&\quad + O(e^{-C/\e})\(1+\sum_\pm |\nabla\uu_0^\pm |^2\1_{\O_0^\pm}\)\qquad \text{a.e. in }\O.\end{align}
Similar but easier calculation of \eqref{F init expand1} yields
 \begin{align}\label{grad uuin6}
(1-\eta_\delta ) F(\uu_\e^{in}) &=  O(e^{-C/\e}) \qquad \text{  in }\O.\end{align}
By \eqref{uu initial}, the Lipschitz continuity of $\dd_F$ and \eqref{def Se1},
\begin{align}\label{dfexpand1}
\dd_F\circ\uu_\e^{in}& =\dd_F\(\tfrac {\uu_0^+ +\uu_0^- }2 +S_\e \tfrac{\uu_0^+ -\uu_0^- }{\dist_\mm}\)\nonumber\\
& =\dd_F\(\tfrac {\uu_0^+ +\uu_0^- }2 +\alpha \tfrac{\uu_0^+ -\uu_0^- }{\dist_\mm}\)+O(e^{-C/\e})\text{ in }\Omega.
\end{align}
This combined with \eqref{exp alpha1} implies that  
\begin{align*}
\dd_F\circ\uu_\e^{in}& =\dd_F\(\tfrac {\uu_0^+ +\uu_0^- }2 +\tfrac{\dist_\mm}2 (\1_{\O_0^+}-\1_{\O_0^-}) \tfrac{\uu_0^+ -\uu_0^- }{\dist_\mm}\)+O(e^{-C/\e})\nonumber\\
&=\dd_F\(\uu_0^+ \1_{\O_0^+}+ \uu_0^- \1_{\O_0^-}\)+O(e^{-C/\e})\qquad \text{ in }\O\backslash B_{\delta/2}(\Sigma_0).
\end{align*}
As $\uu_0^\pm$ are mappings  into $\mm_\pm$ (cf. \eqref{u in extension}), we have from \eqref{quasidistance} that 
\begin{align}\label{grad uuin7}
\dd_F\circ\uu_\e^{in} =c_F \1_{\O_0^+}+O(e^{-C/\e})\quad \text{ in }\O\backslash B_{\delta/2}(\Sigma_0).\end{align}
Since  $(1-\eta_\delta)$ vanishes on $B_{\delta/2}(\Sigma_0)$,
\begin{align}\label{grad uuin8}
\div\((1-\eta_\delta)\bxi\)\dd_F\circ \uu_\e^{in}  
\overset{\eqref{quasidistance}}=\div\((1-\eta_\delta)\bxi\) c_F \1_{\O_0^+} +O(e^{-C/\e}).
\end{align}
So we have 
\begin{align}\label{grad uuin9}
-&\int_\O  (1-\eta_\delta)\bxi\cdot\nabla \(\dd_F\circ \uu_\e^{in}\)  \nonumber\\
&\overset{\eqref{bc n and H}}=\int_\O \div\((1-\eta_\delta)\bxi\)\dd_F\circ\uu_\e^{in} \nonumber\\
&\overset{\eqref{grad uuin8}}= \int_{\O_0^+}\div\((1-\eta_\delta)\bxi\) c_F  +O(e^{-C/\e}) \nonumber\\
&\overset{\eqref{grad uuin8}}=  c_F\int_{\Sigma_0} (1-\eta_\delta)\bxi\cdot\nu \,d\mathcal{H}^{d-1}  +O(e^{-C/\e})=O(e^{-C/\e}).
\end{align}
Putting \eqref{grad uuin5}, \eqref{grad uuin6} and \eqref{grad uuin9} together, we obtain
\begin{align}\label{grad uuin10}
&\int_\O \(\frac 12 |\nabla\uu_\e^{in}|^2+\frac{F(\uu_\e^{in})}{\e^2}- \frac 1 \e \bxi\cdot\nabla(\dd_F \circ \uu_\e^{in})\)(1-\eta_\delta)\nonumber\\
&=\sum_\pm\int_{\O_0^\pm} \frac{1-\eta_\delta}2 |\nabla\uu_0^\pm|^2     +O(e^{-C/\e})+O(e^{-C/\e})\sum_\pm\int_{\O_0^\pm}    |\nabla\uu_0^\pm|^2.\end{align}
Combining this with \eqref{grad uuin3}  leads to
\begin{align}\label{grad uuin11}
&\int_\O \(\frac 12 |\nabla\uu_\e^{in}|^2+\frac{F(\uu_\e^{in})}{\e^2}- \frac 1 \e \bxi\cdot\nabla(\dd_F \circ \uu_\e^{in})\)  \nonumber\\
&=\int_\O\frac 12  \eta_\delta \left|\tfrac{\nabla\uu_0^++\nabla\uu_0^-}2+\alpha(\tfrac{\dd_\Sigma}{\e}) \tfrac{\nabla\uu_0^+ -\nabla\uu_0^-}{\dist_\mm}\right|^2+\sum_\pm\int_{\O_0^\pm} \frac{1-\eta_\delta}2 |\nabla\uu_0^\pm|^2+O(\e) .\end{align}
This leads to \eqref{u cali}.
Using  $\alpha(\tfrac{\dd_\Sigma}{\e})\xrightarrow{\e\to 0} \tfrac{\dist_\mm}2 (\1_{\O_0^+}-\1_{\O_0^-})$ a.e. in $\O$, we can apply the dominated convergence to the first integral on the RHS of \eqref{grad uuin11} and get
\begin{align}\label{grad uuin12}
 \lim_{\e\to 0}\int \(\frac 12 |\nabla\uu_\e^{in}|^2+\frac{F(\uu_\e^{in})}{\e^2}- \frac 1 \e \bxi\cdot\nabla(\dd_F \circ \uu_\e^{in})\)   
 = \sum_\pm \int_{\O_0^\pm}    |\nabla\uu_0^\pm|^2 .\end{align}
 
 \subsection{Proof of \eqref{u bulk}.} Recall  from \eqref{gronwall2new} that 
\begin{align}
 B[\uu_\e^{in}  | \Sigma_0] := &\int_\O   \Big(c_F\chi-c_F+ 2\(\dd_F \circ \uu_\e^{in}-c_F\)^- \Big)\eta\circ \dd_\Sigma  \, dx\nonumber\\
 &+\int_\O  \( \dd_F \circ \uu_\e^{in}-c_F\)^+|\eta\circ\dd_\Sigma| \, dx,\label{gronwall2newint}
\end{align}
 where $\chi =\1_{\O_0^+}-\1_{\O_0^-}$,  and $\eta$ is defined by \eqref{truncation eta}. 
 We also recall from   \eqref{linpanwang 2.2} that 
$c_F =2 \int_{0}^{\tfrac{\dist_{\mm}}2} \sqrt{2\widetilde{F}(\lambda)} d \lambda.$
Concerning the first   integral of \eqref{gronwall2newint},
we first deduce from  \eqref{grad uuin7} that its integrand   is of order $O(e^{-C/\e})$ on $\O\backslash B_{\delta/2}(\Sigma_0)$. So it suffices  to estimate the integral in  the transitional region $B_{\delta}(\Sigma_0)$: by \eqref{dfexpand1}, \eqref{longcase1} and   a change of variable
\begin{align*}
&\int_{\O_0^+\cap B_{\delta}(\Sigma_0)}   \Big(c_F\chi-c_F+ 2\(\dd_F \circ \uu_\e^{in}-c_F\)^- \Big)\eta\circ \dd_\Sigma  \, dx\nonumber\\
 =& ~2\e \int_{\O_0^+\cap B_{\delta}(\Sigma_0)} \(  \int_{\alpha(\tfrac{\dd_\Sigma}{\e})}^{\frac{\dist_\mm}2}\sqrt{2\tilde{F}(\lambda)}\, d\lambda\)\frac{\eta\circ \dd_\Sigma}\e  \, dx+O(e^{-C/\e}) \leq C\e.
\end{align*}
  In a similar way
\begin{align*}
&\int_{\O_0^-\cap B_{\delta}(\Sigma_0)}   \Big(c_F\chi-c_F+ 2\(\dd_F \circ \uu_\e^{in}-c_F\)^- \Big)\eta\circ \dd_\Sigma  \, dx\nonumber\\
 =&~  -2\e\int_{\O_0^-\cap B_{\delta}(\Sigma_0)} \(   \int^{\alpha(\tfrac{\dd_\Sigma}{\e})}_{-\frac{\dist_\mm}2}\sqrt{2\tilde{F}(\lambda)}\, d\lambda\)\frac{\eta\circ \dd_\Sigma}\e  \, dx+O(e^{-C/\e}) \leq C\e.
\end{align*}
Similar calculation shows that the second integral of \eqref{gronwall2newint} is of order $O(\e)$.  
All together we finish the proof of   \eqref{u bulk}.

%%%%%%%%%%%%%%%%%%%%%%%%%%%%%%%%%%%%%%%%%%%%%%%%%%%%%%%%%%%%%%%%%%%%%%%%%%%%%%%%%%%%%%%%%%%%%%%%%%%%%%%%%%%%%%
%%%%%%%%%%%%%%%%%%%%%%

 \ \

 \noindent{\it Acknowledgements}.   
  Y. Liu is partially supported by NSF of China under Grant  11971314.
We would like to thank professor Wei Wang for sharing with us the  notes \cite{wangnote} and stimulating discussions. 
%\noindent {\bf Conflict of interest} The authors declare that they have no conflict of interest.

 \appendix
  
   %%%%%%%%%%%%%%%%%%%%%%%%%%%%%%%%%%%%%%%%%
 \section{Proof of Proposition \ref{gronwallprop}}\label{appendix}

As we shall not integrate the time variable $t$ throughout this section,   we shall abbreviate the spatial integration $\int_\O$ by $\int$ and sometimes we omit the $\,dx$. 
%We shall also employ   Einstein's summation convention by summing over repeated Latin indices.

\begin{lemma}\label{lemma:expansion 1} The following identity holds 	\begin{align}
	&\int \nabla \HH: (\bxi  \otimes\nn_\e)\left|\nabla \psi_\e\right| \, d x
	-\int (\nabla \cdot \HH) \, \bxi   \cdot \nabla \psi_\e \, d x\nonumber\\
	=&\int \nabla \HH: (\bxi -\nn_\e ) \otimes\nn_\e\left|\nabla \psi_\e\right|\, d x+\int \HH_\e\cdot\HH |\nabla \uu_\e |\, d x \nonumber\\
	&+\int\nabla\cdot\HH  \( \frac{\e }2 |\nabla \uu_\e |^2  +\frac{1}\e  F (\uu_\e ) -|\nabla \psi_\e| \)\, d x +\int\nabla\cdot\HH ( |\nabla\psi_\e|-\bxi \cdot\nabla\psi_\e)\, d x\nonumber\\
	&-\sum_{i,j=1}^d\int    (\nabla \HH)_{ij} \,\e  \(\p_i \uu_\e  \cdot \p_j \uu_\e    \)\, d x +\int \nabla \HH: (\nn_\e  \otimes\nn_\e)\left|\nabla \psi_\e\right|\, d x.\label{expansion1}
	\end{align}
\end{lemma}
\begin{proof}
%[Proof of Lemma \ref{lemma:expansion 1}]
	We
	introduce the energy stress tensor $(T_\e)_{ij}= \( \frac{\e }2 |\nabla \uu_\e |^2 +\frac{1}{\e } F (\uu_\e )  \) \delta_{ij} - \e  \p_i \uu_\e \cdot \p_j \uu_\e$.
By \eqref{mean curvature app},
	we have the identity
	$\nabla \cdot T_\e
	   =\HH_\e |\nabla \uu_\e |.$
	Testing this identity by $\HH$, integrating by parts and using \eqref{bc n and H}, we obtain
	\begin{equation*}
	\begin{split}
	&\int \HH_\e\cdot\HH |\nabla \uu_\e |\,d x  =- \int \nabla\HH \colon T_\e \,d x\\
	&=-  \int\nabla\cdot\HH  \( \frac{\e }2 |\nabla \uu_\e |^2 +\frac{1}{\e } F (\uu_\e )  \)\, dx+ \sum_{i,j}\int  (\nabla\HH)_{ij} \, \e \(\p_i \uu_\e \cdot \p_j \uu_\e \) d x.
		\end{split}
	\end{equation*}
	So adding zero leads to
	\begin{equation*}
	\begin{split}
	&\int \nabla \HH: \nn_\e  \otimes\nn_\e\left|\nabla \psi_\e\right|d x\\
	&=\int \HH_\e\cdot\HH |\nabla \uu_\e |\, dx+\int\nabla\cdot\HH  \( \frac{\e }2 |\nabla \uu_\e |^2 +\frac{1}{\e } F (\uu_\e ) -|\nabla \psi_\e| \)\,d x+\int\nabla\cdot\HH |\nabla\psi_\e|\,d x\\
	&-\sum_{i,j}\int  (\nabla\HH)_{ij}\,\e \(\p_i \uu_\e \cdot \p_j \uu_\e    \) d x+\int (\nabla \HH): (\nn_\e  \otimes\nn_\e)\left|\nabla \psi_\e\right|d x.
	\end{split}
	\end{equation*}
	This easily implies   \eqref{expansion1}.
\end{proof}  
The second lemma gives the expansion of the time derivative of \eqref{entropy}.
 \begin{lemma}\label{lemma exact dt relative entropy}
	Under the assumptions of Theorem \ref{main thm}, the following identity holds
\begin{subequations}\label{time deri 4}
		\begin{align}
		\frac{d}{d t} E& [\uu_\e   | \Sigma ]
		+\frac 1{2\e }\int \(\e ^2 \left| \p_t \uu_\e    \right|^2-|\HH_\e |^2\)\,dx\nonumber\\
		&+\frac 1{2\e }\int \Big| \e  \p_t \uu_\e    -(\nabla \cdot \bxi ) \p  \dd_F   (\uu_\e  )  \Big|^2d x
		+\frac 1{2\e }\int \Big| \HH_\e -\e |\nabla \uu_\e  | \HH \Big|^2\,d x\nonumber \\
		=&\frac 1{2\e } \int \Big| (\nabla \cdot \bxi ) |\p  \dd_F   (\uu_\e  )|\nn_\e  +\e  |\Pi_{\uu_\e  } \nabla \uu_\e  | \HH\Big|^2\,d x\label{tail1}
		\\&+\frac \e {2} \int |\HH|^2\(|\nabla \uu_\e  |^2-|\Pi_{\uu_\e  }\nabla \uu_\e  |^2\)\,d x
		-\int \nabla \HH\cdot (\bxi -\nn_\e )^{\otimes 2}\left|\nabla \psi_\e\right|\,d x\label{tail2}\\
		&   +\int \(\nabla\cdot\HH\)  \( \frac{\e }2 |\nabla \uu_\e  |^2 +\frac{1}\e  F  (\uu_\e  ) -|\nabla \psi_\e | \)\,d x\label{tail4}\\
		&+\int\(\nabla\cdot\HH\)  \(1-\bxi \cdot \nn_\e \)|\nabla\psi_\e |\, d x+ \int (J_\e^1+J_\e^2)\, d x,\label{tail3}
		\end{align}
	\end{subequations}
	where $J_\e^1, J_\e^2$ are given by 
	\begin{align}
	  J_\e^1
	:=&  \nabla \HH: \nn_\e  \otimes\nn_\e\(|\nabla \psi_\e |-\e  |\nabla \uu_\e  |^2\)\nonumber\\
	&+\e    \nabla \HH:(\nn_\e \otimes \nn_\e )\(
	|\nabla \uu_\e  |^2-|\Pi_{\uu_\e  } \nabla \uu_\e  |^2\)  \nonumber\\
	&-\sum_{i,j}\e    (\nabla \HH)_{ij}    (\p_i \uu_\e  -\Pi_{\uu_\e  } \p_i \uu_\e  )\cdot(\p_j \uu_\e  -\Pi_{\uu_\e  } \p_j \uu_\e  ), \,  \label{J1}\\
	J_\e^2:= &-  \(\p_t  \bxi +\left(\HH \cdot \nabla\right) \bxi +\left(\nabla \HH\right)^{\mathsf{T}} \bxi \)\cdot \nabla \psi_\e. \label{J2}
	%\nonumber \\
	%&-  \Big(\p_t  \bxi +\left(\HH \cdot \nabla\right) \bxi \Big)\cdot \bxi \, |\nabla \psi_\e |.
	\end{align}
\end{lemma}
\begin{proof}[Proof of Lemma \ref{lemma exact dt relative entropy}]
The proof here  is exactly the same as  in \cite[Lemma 4.4]{MR4284534}.  Note that in the statement of this lemma, the term $\frac 1{2\e }\int \(\e ^2 \left| \p_t \uu_\e    \right|^2-|\HH_\e |^2\)\,dx$  is missing, but  the proof of the identity there is correct (cf. see \cite[equation (4.33)]{MR4284534}).

We shall   employ the  Einstein summation convention by summing over repeated Latin indices.
	Using the energy dissipation law  \eqref{dissipation} and adding zero, we compute the time derivative of the energy \eqref{entropy} by
	\begin{align}
	&\frac{d}{d t} E_\e  [ \uu_\e   | \Sigma]
	+\e \int |\p_t \uu_\e  |^2\,d x-\int (\nabla \cdot \bxi )   \p  \dd_F   (\uu_\e  )\cdot \p_t \uu_\e   \,d x\nonumber\\
	=&\int   \left(\HH  \cdot \nabla\right) \bxi \cdot\nabla \psi_\e \,d x
	+\int \left(\nabla \HH \right)^{\mathsf{T}} \bxi  \cdot\nabla \psi_\e \,d x+ \int J_\e^2\, d x.
	%\nonumber\\
	%&-\int \(\p_t  \bxi +\left(\HH  \cdot \nabla\right) \bxi  
	%+\left(\nabla \HH \right)^{\mathsf{T}} \bxi \)\cdot\nabla \psi_\e \,d x
	 \label{time deri 1}
	\end{align}%Note that the last integral can be handled by \eqref{xi der}.
	Due to the symmetry of the Hessian of $\psi_\e $  and the boundary conditions \eqref{bc n and H}, we have
	\begin{align*}
	\int \nabla \cdot (\bxi  \otimes \HH ) \cdot \nabla \psi_\e  \, d x =  \int \nabla \cdot (\HH \otimes \bxi ) \cdot \nabla \psi_\e  \, d x.
	\end{align*}
	Hence, the first integral on the right-hand side  of \eqref{time deri 1} can be rewritten as
	\begin{align*}
	&\int\left(\HH  \cdot \nabla\right) \,\bxi  \cdot \nabla \psi_\e \, d x\nonumber \\
	&=\int \nabla \cdot (\bxi  \otimes \HH ) \cdot \nabla \psi_\e  \, d x
	-\int (\nabla \cdot \HH ) \,\bxi  \cdot \nabla \psi_\e \, d x\nonumber \\
	&= \int(\nabla \cdot \bxi ) \,\HH  \cdot \nabla \psi_\e \, d x
	+\int(\bxi  \cdot \nabla) \,\HH  \cdot \nabla \psi_\e \, d x
	-\int (\nabla \cdot \HH ) \,\bxi   \cdot \nabla \psi_\e \,d x.
	\end{align*}
	Therefore
	\begin{equation*}%\label{time deri 2}
	\begin{split}
	&\frac{d}{d t} E_\e  [ \uu_\e   | \Sigma]
	+\e \int |\p_t \uu_\e  |^2\,d x-\int (\nabla \cdot \bxi )   \p  \dd_F   (\uu_\e  )\cdot \p_t \uu_\e  \,d x \\
	=& \int (\nabla \cdot \bxi )\, \HH  \cdot \nabla \psi_\e  d x+\int (\bxi  \cdot \nabla) \,\HH  \cdot \nabla \psi_\e \,d x
	-\int (\nabla \cdot \HH )\, \bxi   \cdot \nabla \psi_\e \,d x\\
	&   +\int \nabla \HH : \(\bxi  \otimes\nn_\e\)\left|\nabla \psi_\e\right| d x+ \int J_\e^2\, d x.
	\end{split}
	\end{equation*}
	Using   \eqref{expansion1}   to replace the 3nd and 4th integrals on the right-hand side above yields	\begin{align}\label{time deri 3-}
&\frac{d}{d t} E_\e  [ \uu_\e   | \Sigma]
	+\e \int |\p_t \uu_\e  |^2\,d x-\int  (\nabla \cdot \bxi )  \p  \dd_F   (\uu_\e  )\cdot \p_t \uu_\e  \,d x
	\\   =& \int (\nabla \cdot \bxi ) \,\HH   \cdot \nabla \psi_\e\,d x
	+\int (\bxi  \cdot \nabla)\, \HH  \cdot \nabla \psi_\e\,d x
	+\int \nabla \HH : (\bxi -\nn_\e  ) \otimes\nn_\e\left|\nabla \psi_\e\right|\,d x\nonumber\\
	&   +\int \HH_\e \cdot \HH |\nabla \uu_\e  |\,d x
	+ \int \nabla\cdot\HH  \( \frac{\e }2 |\nabla \uu_\e  |^2 +\frac{1}\e  F  (\uu_\e  ) -|\nabla \psi_\e | \)\,d x\nonumber\\&
	+\int \nabla\cdot\HH \(|\nabla\psi_\e |-\bxi \cdot \nabla\psi_\e \)\,d x-\int     (\nabla \HH)_{ij} \,\e  \(\p_i \uu_\e  \cdot \p_j \uu_\e    \)\, d x\nonumber\\
	&
	+\int \nabla \HH : \nn_\e   \otimes\nn_\e\left|\nabla \psi_\e\right|\,d x   +\int J_\e^2\, dx.\nonumber
	\end{align}
We claim that   $J_\e^1$ arises from the 2nd and 3rd to last integral. Indeed, when $\p \dd_F(\uu_\e )\neq 0$ and $\uu_\e\notin B_{\delta_0}(\mm)$,	then it follows from \eqref{projection1} and \eqref{ADM chain rule} that 
\begin{align*}
\Pi_{\uu_\e  } \p_i \uu_\e   \cdot \Pi_{\uu_\e  } \p_j \uu_\e |\p\dd_F(\uu_\e)|^2 =\p_i \psi_\e  \p_j \psi_\e \overset{\eqref{projectionnorm}}=n_\e^i n_\e^j |\Pi_{\uu_\e  } \nabla \uu_\e|^2|\p\dd_F(\uu_\e)|^2,
\end{align*}
where $(n_\e^i)_{1\leq i\leq d}=\nn_\e$.
This implies 
\[\Pi_{\uu_\e  } \p_i \uu_\e   \cdot \Pi_{\uu_\e  } \p_j \uu_\e=n_\e^i n_\e^j|\Pi_{\uu_\e  } \nabla \uu_\e|^2.\label{key identity chain}\]
In other   cases defining \eqref{projection1}, the   equation \eqref{key identity chain} also holds.  Using the orthogonality of  \eqref{projection1}, adding  zero and using \eqref{key identity chain}, we find  
	\begin{equation*}
	\begin{split}
	&   \nabla \HH : \nn_\e   \otimes\nn_\e\left|\nabla \psi_\e\right|-   (\nabla \HH)_{ij} \,\e  \(\p_i \uu_\e  \cdot \p_j \uu_\e    \)\\
	 =&  \nabla \HH : \nn_\e   \otimes\nn_\e\left|\nabla \psi_\e\right|-   \e (\nabla \HH)_{ij}(\Pi_{\uu_\e  } \p_i \uu_\e   \cdot \Pi_{\uu_\e  } \p_j \uu_\e  ) \\
	& -   (\nabla \HH)_{ij} \,\e  \Big((\p_i \uu_\e  -\Pi_{\uu_\e  } \p_i \uu_\e  )\cdot(\p_j \uu_\e  -\Pi_{\uu_\e  } \p_j \uu_\e  )\Big)  =  J_\e^1,
%	&\overset{\eqref{projection}}=  \nabla \HH : \nn_\e   \otimes\nn_\e\(|\nabla \psi_\e |-\e  |\nabla \uu_\e  |^2\)+   \e \nabla \HH:(\nn_\e  \otimes \nn_\e  )\(
%	|\nabla \uu_\e  |^2-|\Pi_{\uu_\e  } \nabla \uu_\e  |^2\) \\
%	&\quad-    (\nabla \HH)_{ij} \,\e  \Big((\p_i \uu_\e  -\Pi_{\uu_\e  } \p_i \uu_\e  )\cdot(\p_j \uu_\e  -\Pi_{\uu_\e  } \p_j \uu_\e  )\Big)  =  J_\e^1.
	\end{split}
	\end{equation*}
	and this finish the proof of the claim.
%On the set $\{ x \mid \p \dd_F(\uu_\e )= 0\}$, by \eqref{projection1} and \eqref{ADM chain rule} we have  $\Pi_{\uu_\e  } \p_j \uu_\e=0$ and $\nabla \psi_\e=0$  a.e..   So  we obtain the above identity without using  \eqref{projection}. 

Now we write   the sum of the integrands of  2nd and  3rd  integrals on the right-hand side  of \eqref{time deri 3-} into a quadratic term of the difference: using the definition \eqref{normal diff} of $\nn_\e $ and $\nabla \HH   :(\bxi \otimes \bxi )=0$ (due to \eqref{def:xi} and \eqref{normal H}), we have 
\begin{align*}
&(\bxi  \cdot \nabla)\, \HH  \cdot \nabla \psi_\e+\nabla \HH : (\bxi -\nn_\e  ) \otimes\nn_\e |\nabla \psi_\e |\\
=&(\bxi  \cdot \nabla)\, \HH  \cdot \nn_\e |\nabla \psi_\e|+\nabla \HH : (\bxi -\nn_\e  ) \otimes\nn_\e |\nabla \psi_\e |-\nabla \HH   :(\bxi \otimes \bxi )|\nabla \psi_\e |\\
=& \nabla \HH  : (\nn_\e\otimes  \bxi) |\nabla \psi_\e|+\nabla \HH : (\bxi -\nn_\e  ) \otimes\nn_\e |\nabla \psi_\e |-\nabla \HH   :(\bxi \otimes \bxi )|\nabla \psi_\e |\\
=&-\nabla \HH : (\bxi -\nn_\e  )^{\otimes 2}\left|\nabla \psi_\e\right|.
\end{align*}
Using this identity, we can   merge  the 2nd and  3rd  integrals on the right-hand side  of \eqref{time deri 3-}:
	\begin{align}\label{time deri 3}
	&\frac{d}{d t} E_\e  [ \uu_\e   | \Sigma]
	=-\e \int |\p_t \uu_\e  |^2\, dx+\int  (\nabla \cdot \bxi )  \p  \dd_F   (\uu_\e  )\cdot \p_t \uu_\e  \, dx\nonumber\\
	&+ \int (\nabla\cdot \bxi ) \,\HH   \cdot \nabla \psi_\e\, dx+\int \HH_\e \cdot \HH |\nabla \uu_\e  | \, dx -\int \nabla \HH : (\bxi -\nn_\e  )^{\otimes 2}\left|\nabla \psi_\e\right|\, dx\nonumber\\
	&   +\int \(\nabla\cdot\HH\)  \Big( \frac{\e }2 |\nabla \uu_\e  |^2 +\frac{1}\e  F  (\uu_\e  ) -|\nabla \psi_\e | \Big)\, dx\nonumber\\&+\int (\nabla\cdot\HH) \(1-\bxi \cdot \nn_\e  \)|\nabla\psi_\e |\, dx+ \int (J_\e^1+J_\e^2) \, dx.
	\end{align}

	%Note that  the second and the third integrals in the last step above can  be controlled  using  \eqref{energy bound1} and \eqref{energy bound2}.
	
	Now we complete squares for the first four terms on the right-hand side  of \eqref{time deri 3}.
	Reordering terms, we have
	\begin{align*}
	\notag-&\e  |\p_t \uu_\e  |^2+   (\nabla \cdot \bxi )  \p  \dd_F   (\uu_\e  )\cdot \p_t \uu_\e  
	+ (\nabla \cdot \bxi ) \HH   \cdot \nabla \psi_\e+ \HH_\e \cdot \HH |\nabla \uu_\e  |
	\\\notag&= -\frac1{2\e } \Big(  |\e  \p_t \uu_\e  |^2 -2(\nabla \cdot \bxi )  \p  \dd_F   (\uu_\e  )\cdot \e  \p_t \uu_\e  
	+(\nabla \cdot \bxi )^2 | \p  \dd_F   (\uu_\e  )|^2 \Big)
	\\\notag&\quad - \frac1{2\e } |\e  \p_t \uu_\e  |^2 + \frac1{2\e }(\nabla \cdot \bxi )^2 | \p  \dd_F   (\uu_\e  )|^2
	+ (\nabla \cdot \bxi ) \HH   \cdot \nabla \psi_\e
	\\\notag&\quad - \frac1{2\e } \Big( |\HH_\e |^2 - 2 \e  |\nabla \uu_\e  | \HH_\e \cdot  \HH + \e ^2 |\nabla \uu_\e  |^2 |\HH|^2\Big)
	+ \frac1{2\e } \Big( |\HH_\e |^2 + \e ^2 |\nabla \uu_\e  |^2 |\HH|^2\Big)
	\\&\notag =  -\frac1{2\e } \Big|\e  \p_t \uu_\e  - (\nabla \cdot \bxi ) \p  \dd_F   (\uu_\e  ) \Big|^2
	- \frac1{2\e } \Big|\HH_\e  - \e  |\nabla \uu_\e  | \HH \Big|^2
	- \frac1{2\e }  |\e  \p_t \uu_\e  |^2 +\frac1{2\e }  |\HH_\e |^2
	\\&\quad + \frac1{2\e } \Big( (\nabla \cdot \bxi )^2  |\p  \dd_F   (\uu_\e  )|^2 + 2\e (\nabla \cdot \bxi )  \nabla \psi_\e  \cdot \HH + |\e  \Pi_{\uu_\e  }\nabla \uu_\e  |^2 |\HH|^2 \Big)
	\\\notag&\quad +\frac\e {2} \left(|\nabla \uu_\e  |^2- | \Pi_{\uu_\e  }\nabla \uu_\e  |^2\right)  |\HH|^2.
	\end{align*}
	Using  \eqref{normal diff}   and the chain rule \eqref{projectionnorm}, the terms above  form the last missing square.
	Integrating over the domain $\Omega$ and substituting into \eqref{time deri 3} we arrive at \eqref{time deri 4}.
	 
\end{proof}

\begin{proof}[Proof of Proposition \ref{gronwallprop}]
	We first  estimate the right-hand side  of \eqref{time deri 4} by $E_\e  [\uu_\e |\Sigma]$ up to a constant that only depends on $\Sigma_t$. We start with \eqref{tail1}: it follows from the  triangle inequality that
	\begin{equation*}
	\begin{split}
	 \int& \left|\e^{-1/2} (\nabla \cdot \bxi ) |\p  \dd_F   (\uu_\e  )|\nn_\e  +\sqrt{\e } |\Pi_{\uu_\e  } \nabla \uu_\e  | \HH\right|^2d x
\\	\leq &\int \left|(\nabla\cdot \bxi )
	\left(	   \e^{-1/2} |\p  \dd_F   (\uu_\e  )|-\sqrt{\e } |\Pi_{\uu_\e  } \nabla \uu_\e  |
	\right)\nn_\e \right|^2\,d x\\&+\int \Big|(\nabla\cdot \bxi )\sqrt{\e }  |\Pi_{\uu_\e  } \nabla \uu_\e  |(\nn_\e -\bxi )\Big|^2\,	d x	\\&+ \int \left|
	\big((\nabla \cdot \bxi ) \bxi  +\HH \big)  \sqrt{\e } |\Pi_{\uu_\e  } \nabla \uu_\e  | \right|^2\,d x.
	\end{split}
	\end{equation*}
	The first  integral on the right-hand side  of the above inequality is controlled by \eqref{energy bound2}.  Due to the elementary inequality  $|\bxi  - \nn_\e |^2 \leq 2 (1-\nn_\e \cdot\bxi )$,   the second integral is controlled by \eqref{energy bound1}. The third integral can be treated using  the relation $\HH=(\HH\cdot\bxi ) \bxi +O(\dd_\Sigma(x,t))$ and \eqref{div xi H}. So it can be  controlled by \eqref{energy bound3}.

	The integrals in \eqref{tail2} can be controlled using \eqref{energy bound0} and \eqref{energy bound1}. The integrals in \eqref{tail4} is controlled by \eqref{energy bound-1}.
The first term in \eqref{tail3} can be controlled using \eqref{energy bound1}.  It remains to estimate   \eqref{J1} and \eqref{J2}. The last two terms defining  $J_\e^1$ can be bounded using   \eqref{energy bound0}. Therefore,
	\begin{align*}
	\int J_\e^1\, dx
	\overset{\eqref{energy bound0}}\leq &\int \nabla \HH: \( \nn_\e  \otimes (\nn_\e -\bxi )\) \(|\nabla \psi_\e |-\e  |\nabla \uu_\e  |^2\)\, dx\nonumber\\
	&+\int (\bxi \cdot \nabla) \HH\cdot \nn_\e \(|\nabla \psi_\e |-\e  |\nabla \uu_\e  |^2\)\, dx+ C E_\e  [\uu_\e |\Sigma]\nonumber\\
	\overset{\eqref{normal H}}\lesssim & \int  |\nn_\e -\bxi | \Big(\e  |\nabla \uu_\e  |^2-\e  |\Pi_{\uu_\e  }\nabla \uu_\e  |^2\Big)\, dx\\
	&+\int  |\nn_\e -\bxi | \left|\e  |\Pi_{\uu_\e  }\nabla \uu_\e  |^2-|\nabla \psi_\e |\right| \, dx\nonumber\\
	&+ \int\min \(\dd_\Sigma^2 ,1\)  \(|\nabla \psi_\e |+\e  |\nabla \uu_\e  |^2\)\, dx+   E_\e  [\uu_\e |\Sigma].
	\end{align*}
	  The first  and the third  integrals in the last display  can be estimated using  \eqref{energy bound0} and \eqref{energy bound3} respectively. 
	Then  we employ \eqref{projectionnorm} and yield 
	\begin{align*}
	\int J_\e^1\, dx
	 \lesssim &\, \int  |\nn_\e -\bxi | \left| \e  |\Pi_{\uu_\e  }\nabla \uu_\e  |^2-|\nabla \psi_\e |\right|\, dx+   E_\e  [\uu_\e |\Sigma]\nonumber\\
	 = &\,  \int  |\nn_\e -\bxi | \sqrt{\e } |\Pi_{\uu_\e  }\nabla \uu_\e  | \left| \e^{1/2} |\Pi_{\uu_\e  }\nabla \uu_\e  |-\e^{-1/2}|\p  \dd_F   (\uu_\e  )|\right|\, dx+  E_\e  [\uu_\e |\Sigma].
	\end{align*}
	
	Finally applying the Cauchy-Schwarz inequality and then \eqref{energy bound2} and \eqref{energy bound1}, we obtain\\
$\int J_\e^1 \lesssim  E_\e  [\uu_\e |\Sigma].$
	As for $J_\e^2$ \eqref{J2}, we employ \eqref{xi der1} and  \eqref{energy bound3}  to obtain  $\int J_\e^2  \lesssim E_\e  [\uu_\e  | \Sigma].$
  Altogether,  we prove that the right-hand side  of \eqref{time deri 4} is bounded by $E_\e  [\uu_\e |\Sigma]$ up to a multiplicative constant which only depends on $\Sigma_t$.
\end{proof}

%\bibliography{/Users/yl67/Dropbox/cfs.bib}
%\bibliographystyle{abbrv}
  
\end{document}